\documentclass[12pt]{amsart}

\usepackage{amsmath,amssymb,amsthm,amsfonts}
\usepackage{mathtools}
\usepackage[left=1in,textwidth=148mm,textheight=20cm,top=2.5cm]{geometry}
\usepackage[utf8]{inputenc}
\usepackage[english]{babel}
\usepackage{graphicx}
\usepackage{xcolor}
\usepackage{tikz-cd}
\usepackage{enumitem}
\usepackage{framed}
\usepackage{varwidth}
\usepackage{booktabs}
\usepackage{xurl}
\usepackage[colorlinks=true,linkcolor=blue,citecolor=blue,urlcolor=blue]{hyperref}
\usepackage{caption}
\usepackage[backend=biber]{biblatex}
\allowdisplaybreaks[4]
\DeclareMathOperator{\sgn}{sgn}
\DeclareMathOperator{\im}{im}

\newtheorem{Theorem}{Theorem}[section]

\newtheorem{Corollary}[Theorem]{Corollary}

\makeatletter
\def\subsection{\@startsection{subsection}{2}%
  \z@{.7\linespacing\@plus\linespacing}{.5\linespacing}%
  {\normalfont\itshape}} 
\makeatother

\begin{document}

    \begin{abstract}
        Classical Graph Signal Processing (GSP) provides a robust framework for analyzing signals on irregular domains, utilizing the graph Fourier transform as a cornerstone for spectral analysis and filtering. However, as data structures grow in complexity, there is an increasing need to handle multi-dimensional information. In this paper, we propose a generalization of the GSP framework by introducing vector-valued graph signals which take values in arbitrary Banach spaces. We define and investigate the fundamental operators of vertex-frequency analysis within this broader setting, including the Fourier transform, convolution, and translation operators. A key contribution of this work is the derivation of operator norm estimates and the establishment of graph-theoretic versions of classical uncertainty principles. We demonstrate how these results depend on the choice of the orthonormal basis and on the underlying $L^p$ norms. By modeling multiple scalar signals as a single vector-valued entity, this framework facilitates the study of inter-signal correlations, providing a flexible and mathematically grounded environment for analyzing multivariate time-series and time-varying signals on complex networks.
    \end{abstract}

    \title{Fourier transform of vector-valued graph signals}
    \author{Antonio Caputo}
    \address{Università di Torino, Dipartimento di Matematica, via Carlo Alberto 10, 10123 Torino, Italy}
    \email{antonio.caputo@unito.it}
    \maketitle
    

    \section{Introduction}
    In the last decades, graph theory has undergone significant development from both theoretical and practical points of view, spanning different branches of Mathematics, such as Algebra, Geometry, Analysis and Numerical Analysis. It has now become a fundamental branch of Mathematics that models relationships between objects, making it an essential tool for quantifying and simplifying the complex systems of our interconnected world. By representing data as nodes (vertices) and their connections as edges, it provides a universal framework for solving critical problems in computer science, in logistics, in social sciences, and in biology \cite{Bulai_1}.

    Graph theory from an algebraic point of view can be approached using techniques from linear algebra and matrix theory \cite{Biggs, Chung, Wilson}. Indeed, the study of the eigenvalues of graph matrices, such as the adjacency matrix and the Laplacian matrix, reveals some properties of the connectivity of the graph. For example, the second smallest eigenvalue of the Laplacian matrix measures how robustly a network is connected and how easily information can flow through it.

    Throughout this paper, we will work with a finite graph $G=(V,E)$, with \textit{vertex set} $V=\lbrace1,\dots,N\rbrace$ and \textit{edge set} $E$. We assume that $G$ has no loops and that no two distinct edges are incident with the same vertices. The \textit{adjacency matrix} of $G$ is the $N\times N$ matrix $A=(a_{ij})$ defined by
    \begin{equation*}
        a_{ij}=
        \begin{cases}
            1 & \text{if } (i,j)\in E, \\
            0 & \text{if } (i,j)\notin E.
        \end{cases}
    \end{equation*}

    The graph $G$ is said to be \textit{undirected} if $A$ is a symmetric matrix, i.e., if $a_{ij}=a_{ji}$ for all $i,j\in V$.

    The \textit{degree matrix} of $G$ is the diagonal $N\times N$ matrix $D$ whose $i$-th element is the number of vertices of $G$ which are adjacent to $i$. The \textit{Laplacian matrix} of $G$ is the $N\times N$ matrix $L$ defined as $L=D-A$, so that
    \begin{equation*}
        L_{ij}=
        \begin{cases}
            d_i & \text{if } i=j, \\
            -1 & \text{if } (i,j)\in E, \\
            0 & \text{otherwise}.
        \end{cases}
    \end{equation*}

    Finally, the \textit{normalized Laplacian matrix} of $G$ is the $N\times N$ matrix $\mathcal{L}$ defined as $\mathcal{L}=D^{-1/2}LD^{-1/2}$, so that
    \begin{equation*}
        \mathcal{L}_{ij}=
        \begin{cases}
            1 & \text{if } i=j, \\
            \frac{-1}{\sqrt{d_id_j}} & \text{if } (i,j)\in E, \\
            0 & \text{otherwise}.
        \end{cases}
    \end{equation*}

    Notice that $\mathcal{L}$ can be seen as an operator on the space of all functions $f:V\to\mathbb{R}$ satisfying the following relation:
    \begin{equation*}
        \mathcal{L}f(u)=\frac{1}{\sqrt{d_u}}\sum_{v\sim u}\left(\frac{f(u)}{\sqrt{d_u}}-\frac{f(v)}{\sqrt{d_v}}\right),
    \end{equation*}
    where $v\sim u$ means that $(u,v)\in E$, i.e., $u$ and $v$ are adjacent. It is also easy to see that $\mathcal{L}$ satisfies the following equality:
    \begin{equation*}
        \mathcal{L}=I-D^{-1/2}AD^{-1/2},
    \end{equation*}
    where $I$ is the identity matrix of order $N$.

    The set of all functions $f:V\to\mathbb{R}$ is the set of \textit{(real) graph signals} on $G$.

    For example, the \textit{complete graph} is a graph in which any two distinct vertices are adjacent. So, the adjacency matrix $A$ and the Laplacian matrix $L$ of a complete graph with $N=4$ vertices are the following:
    \begin{equation*}
        A=
        \begin{pmatrix}
            0 & 1 & 1 & 1 \\
            1 & 0 & 1 & 1 \\
            1 & 1 & 0 & 1 \\
            1 & 1 & 1 & 0
        \end{pmatrix}, \qquad L=
        \begin{pmatrix}
            3 & -1 & -1 & -1 \\
            -1 & 3 & -1 & -1 \\
            -1 & -1 & 3 & -1 \\
            -1 & -1 & -1 & 3
        \end{pmatrix}.
    \end{equation*}

    In Figure \ref{Examples of graphs} there are some examples of undirected finite graphs.
    \begin{figure}[ht]
        \centering
        \includegraphics[width=0.45\linewidth]{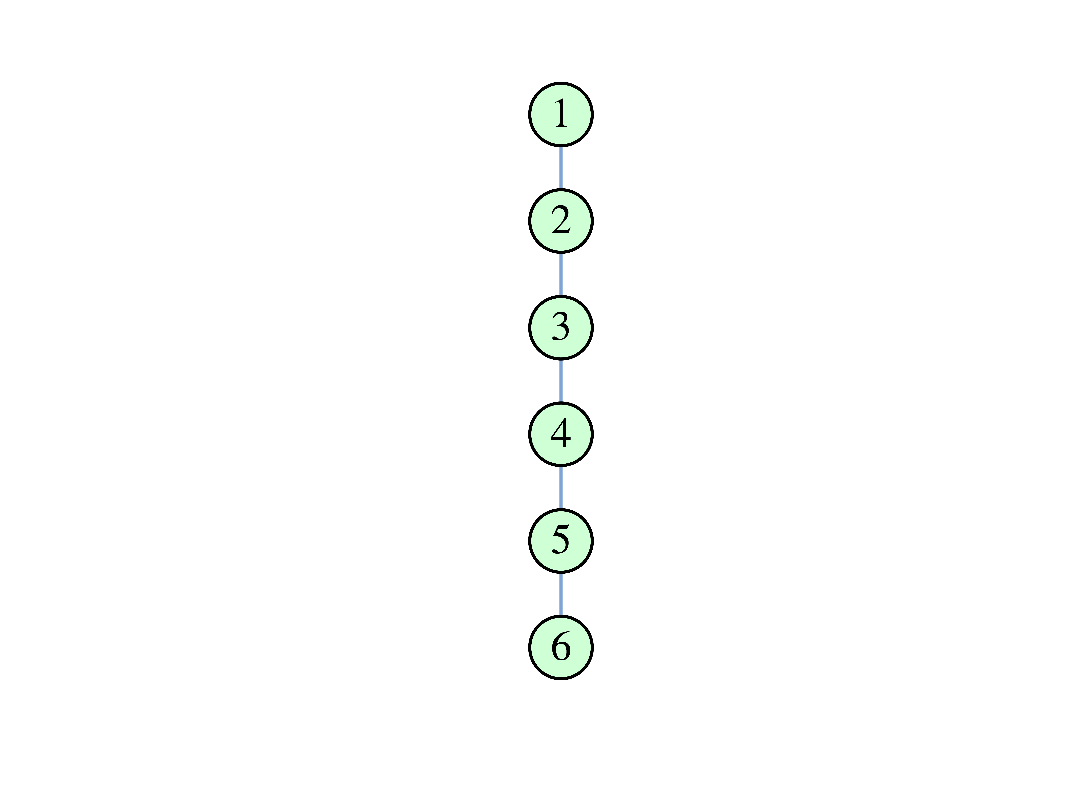}
        \includegraphics[width=0.45\linewidth]{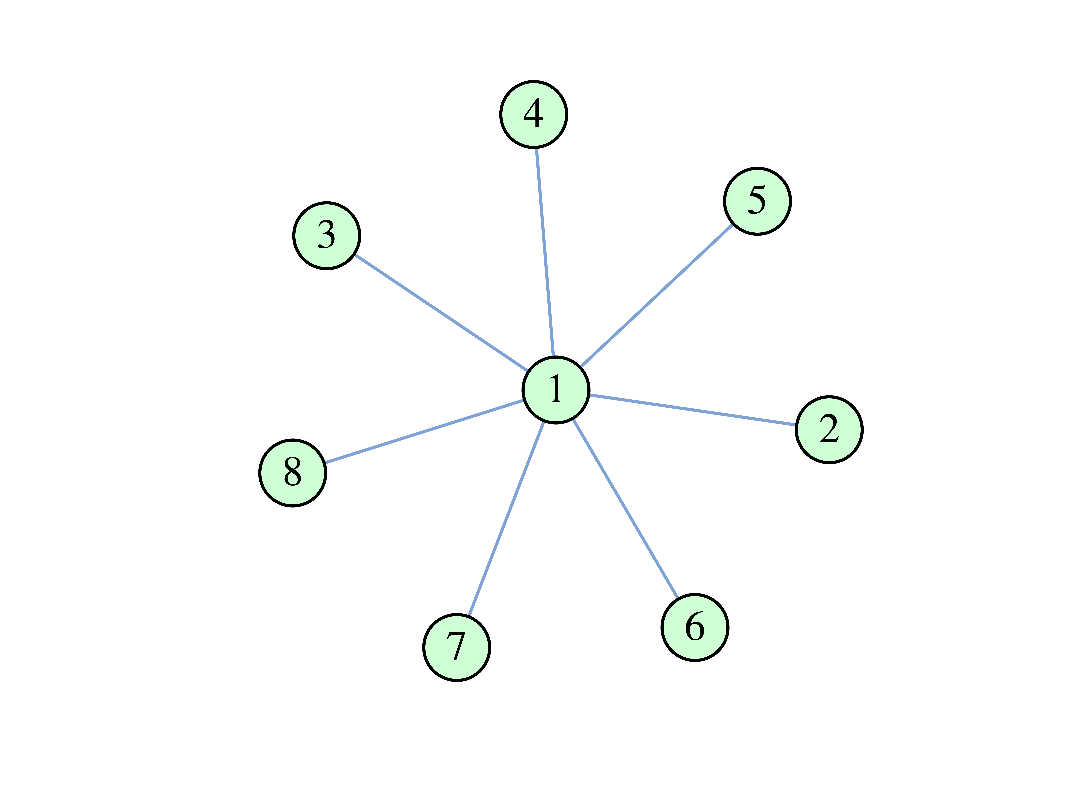} \\
        \includegraphics[width=0.45\linewidth]{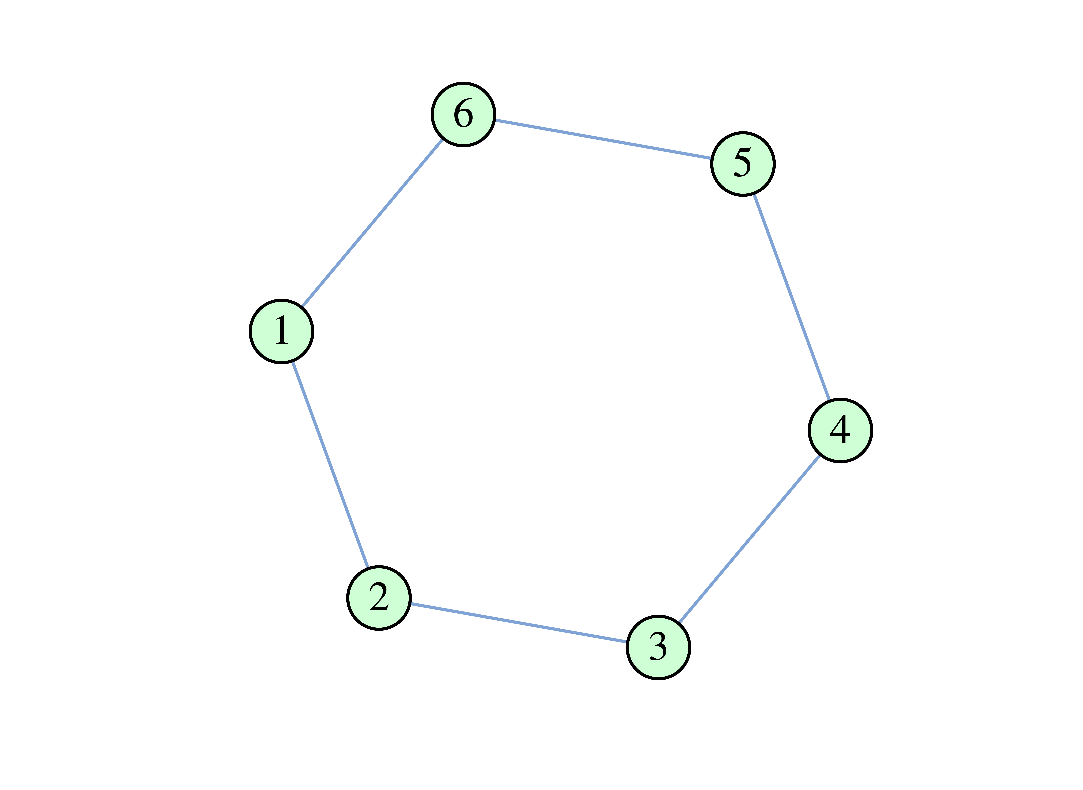}
        \includegraphics[width=0.45\linewidth]{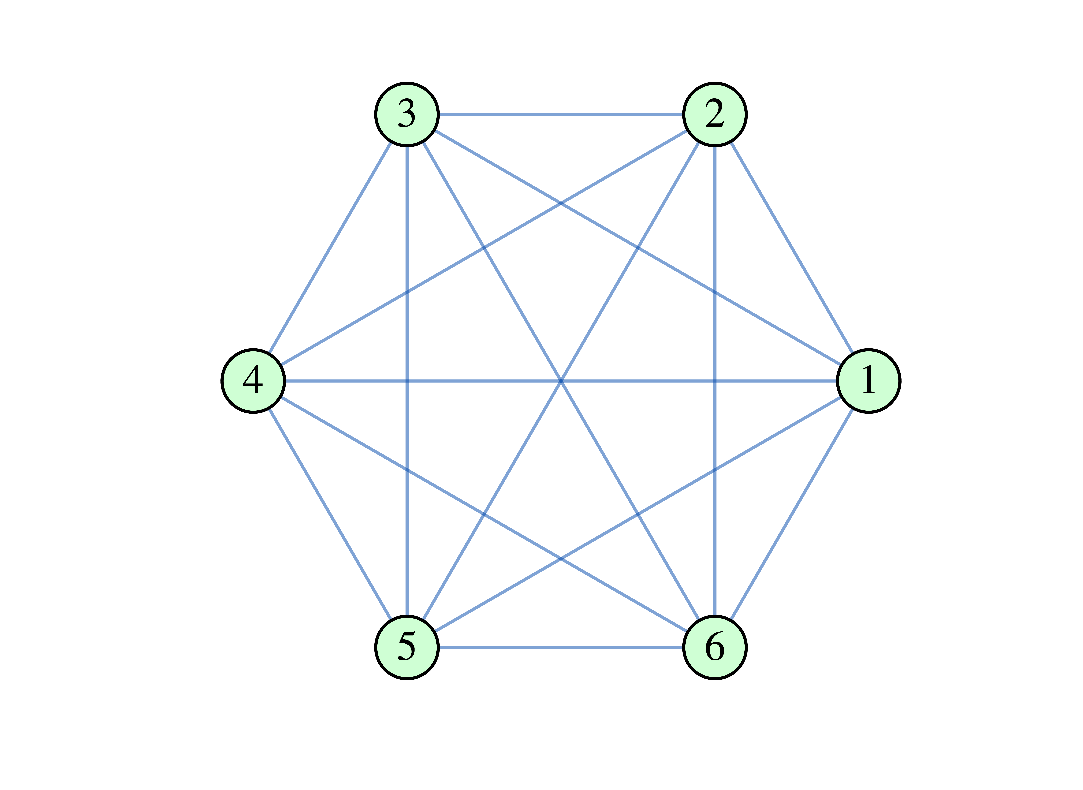}
        \caption{Path graph, star graph, circular graph, complete graph}
        \label{Examples of graphs}
    \end{figure}

    If the adjacency matrix $A$ is a symmetric matrix, the Laplacian matrix $L$ and the normalized Laplacian matrix $\mathcal{L}$ are also symmetric matrices. So, by the spectral theorem, they are diagonalizable by an orthogonal matrix $U$, and their eigenvalues are all real. Moreover, the normalized Laplacian matrix $\mathcal{L}$ is a positive-semidefinite matrix, so its eigenvalues are non-negative and can be ordered as $0\leq\lambda_0\leq\cdots\leq\lambda_{N-1}$. It is also easy to see that $\lambda_0=0$, as $\mathcal{L}$ is singular, and that $\lambda_{N-1}\leq2$. Hence, the spectrum of the matrix $\mathcal{L}$ is localized in the interval $[0,2]$. If $\Lambda$ is the diagonal matrix whose $i$-th element is the $i$-th eigenvalue of the matrix $F\in\lbrace A,L,\mathcal{L}\rbrace$, then we can write its spectral decomposition as the following:
    \begin{equation*}
        F=U\Lambda U^*,
    \end{equation*}
    where $U^*$ denotes the (conjugate) transpose of $U$.
    
    Inspired by the success of classical signal processing, which operates on regular, ordered data (like time-series or grids), Shuman et al. \cite{Shuman_1} defined the classical time-frequency operators, such as convolution, translation, and modulation operators, for signals on graphs, moving towards graph signal processing, which extends these concepts to irregular, non-Euclidean structures represented by graphs. These operators can be adapted to the irregular graph setting via the graph Fourier transform, introduced by Hammond et al. \cite{Hammond} as the expansion of a signal $f:V\to\mathbb{R}$ in terms of the orthonormal basis of eigenvectors of the normalized Laplacian matrix $\mathcal{L}$. In particular, if $\mathcal{B}=\lbrace u_1,\dots,u_N\rbrace\subseteq\mathbb{R}^N$ is such a basis, then the \textit{graph Fourier transform} and the \textit{inverse graph Fourier transform} of a signal $f:V\to\mathbb{R}$ are the signals $\hat{f}:V\to\mathbb{R}$ and $\check{f}:V\to\mathbb{R}$ given by
	\begin{equation*}
		\hat{f}(k)=\sum_{n=1}^{N}u_k(n)f(n), \qquad \check{f}(n)=\sum_{k=1}^{N}u_k(n)f(k).
	\end{equation*}

    It is straightforward to verify that the inverse graph Fourier transform is indeed the inverse operator of the graph Fourier transform. The convolution operator can be defined in the graph spectral domain as
    \begin{equation*}
        (f*g)(n)=\sum_{k=1}^{N}u_k(n)\hat{f}(k)\hat{g}(k).
    \end{equation*}

    In an analogous way, we can define the translation operator $T_m$ of vertex $m\in V$ as the convolution with a delta signal centred at $m$:
    \begin{equation*}
        T_mf(n)=(\delta_m*f)(n)=\sum_{k=1}^{N}u_k(n)u_k(m)\hat{f}(k).
    \end{equation*}

    Many basic properties of the convolution operator still hold in the graph setting. For example, the graph Fourier transform turns the convolution operator of two signals into the pointwise product of their Fourier transforms. Analogously, there are some nice properties of the translation operator which still hold in the graph setting, and some other properties which do not hold. For example, the operator $T_m$ is not always an isometry with respect to the $2$-norms, and this reflects the lack of a sum operation on the vertex set of $G$.

    Based on the previous definitions, Hammond et al. \cite{Hammond} also defined the spectral graph wavelet transform, generated by wavelet operators that are operator-valued functions of the Laplacian, generalizing the classical wavelet transform to the graph setting. Inspired by this approach, Bulai et al. \cite{Bulai_1} generalized the work in \cite{Hammond} and introduced the definition of spectral graph wavelet packets transform in the graph context, extending the classical framework. Moreover, many other classical time-frequency operators have been extended to the graph setting, such as the windowed graph Fourier transform \cite{Shuman_3}, and the multi-windowed graph Fourier transform \cite{Bulai_2}, allowing multiple analysis and synthesis windows. Finally, a foundational overview of Graph Signal Processing can be found in \cite{Stankovic}.

    In a more general way, Sandryhaila et al. \cite{Sandryhaila} defined the graph Fourier transform on the graph $G$ as the expansion of a signal $f$ in terms of the basis of generalized eigenvectors of the adjacency matrix $A$, via its Jordan normal form decomposition, allowing the study of not undirected graphs and of complex graph signals $f:V\to\mathbb{C}$. In this setting, the matrix $A$ is treated as the graph shift operator. In fact, just as a time-delay operator shifts a signal in time, multiplying a signal by the matrix $A$ shifts information to neighbouring nodes.

    If $A$ is not assumed to be a symmetric matrix, as in \cite{Sandryhaila}, diagonalizability is no longer guaranteed, necessitating the use of the Jordan normal form decomposition. To ensure diagonalizability while preserving the energy-saving properties of a unitary transform, the standard condition required is normality. A matrix $A$ is normal if it commutes with its conjugate transpose, i.e., $A^*A=AA^*$, a property defining what is known as a \textit{normal graph}. By the complex spectral theorem, normality is the necessary and sufficient condition for $A$ to be unitarily diagonalizable. Notably, as a normal matrix may have complex eigenvalues, the resulting orthonormal basis of eigenvectors resides in $\mathbb{C}^N$. The fundamental properties of the spectrum of a normal graph are discussed in \cite{Torgasev}.

    In several applications, the choice of the right basis of eigenvectors is crucial, since one of them could be more appropriate than the others in terms of computation or sharpness of the resulting estimates. Hence, Ghandehari et al. \cite{Ghandehari_2, Ghandehari_1}, and Ji et al. \cite{Ji} considered the graph Fourier transform induced by an arbitrary orthonormal basis of $\mathbb{C}^N$. In \cite{Ji}, the set of graph signals is also broadened to include functions taking values into an arbitrary (separable) Hilbert space, showing that the graph Fourier transform can be defined in this setting, also.

    In this work, we follow the approach of considering a general orthonormal basis of $\mathbb{C}^N$, which subsumes the cases of orthonormal eigenbases derived from matrices such as $A$, $L$ or $\mathcal{L}$, and we consider the more general setting of graph signals taking values into an arbitrary Banach space. So, let $\mathbb{K}$ denote either the real or complex field, and let $X$ be a Banach space over $\mathbb{K}$. We consider $X$-valued functions defined on the vertex set $V$ of a finite graph $G=(V,E)$, and we define an \textit{$X$-valued signal} on $G$ as a function $f:V\to X$.

    Since $V$ is finite, the space of $X$-valued signals on $G$ is isomorphic to $X^N$ via the canonical isomorphism
    \begin{equation*}
        \varphi:\lbrace f:V\to X\rbrace\to X^N, \quad f\mapsto\varphi(f):=(f(1),\dots,f(N)).
    \end{equation*}
    
    This identification allows us to treat signals either as functions or elements of the product space $X^N$, so that each component $f(n)$ is an element of $X$. For example, if $X=\mathbb{R}^d$, then $f(n)$ is a $d$-dimensional real vector for all $n\in\lbrace1,\dots,N\rbrace$.

    In this general framework, we establish a solid functional-analytic foundation for the study of vector-valued signals on graphs. Our main contributions include:
    \begin{itemize}
        \item A rigorous formulation of the space of vector-valued signals as standard $L^p$ Bochner spaces and the associated norms.
        \item A detailed analysis of the graph Fourier transform and of translation, modulation, and convolution operators in the context of $X^N$, extending classical graph signal processing tools to vector-valued signals.
        \item The development of new results on operator norms, continuity, and boundedness in the $L^p$ setting, including generalized versions of the Fourier transform on graphs.
    \end{itemize}
    These results pave the way for applying time-frequency techniques and Gabor frame theory to structured graph data, particularly when the data are naturally multidimensional or functional.
    
    The paper is organized as follows. In Section 2, we introduce the graph Fourier transform for vector-valued graph signals, with respect to an arbitrary orthonormal basis of $\mathbb{K}^N$, generalizing the spectral definitions established in \cite{Hammond, Sandryhaila, Shuman_1, Stankovic}, which rely on specific eigenbases of the Laplacian or adjacency matrices, and follows the broader functional-analytic framework proposed in \cite{Ghandehari_2, Ghandehari_1, Ji}. Then, we provide examples of multidimensional and functional signals, and prove estimates involving the operator norm of $\mathcal{F}$. In Section 3, we discuss some simple forms of the uncertainty principle in the graph setting. In Section 4, we define the convolution operator between a scalar signal and a vector-valued signal $f$ via the graph Fourier transform, and study some basic properties, such as Young's inequality on graphs. In Section 5, we define the translation operator for graph signals via convolution and exhibit the condition under which the translation is an isometry. In Section 6, we conclude with some applications and numerical simulations, showing how the estimates found in Section 2 change under a change of basis.

    \section{Operator norms of the graph Fourier transform}
    In this section, we introduce the graph Fourier transform for vector-valued signals and derive some important estimates of its operator norm, as an operator between the Banach spaces $L^p$ and $L^q$, with an arbitrary choice of $p,q\in[1,\infty]$. Moreover, we establish the condition under which a Banach space has hilbertian structure via Plancherel's equality in the graph setting.
    
    \subsection{Preliminaries}
    Let $\mathbb{K}$ denote either the real or complex field, and consider a fixed orthonormal basis $\mathcal{B}=\lbrace u_1,\dots,u_N\rbrace$ of $\mathbb{K}^N$. Throughout this paper, we assume $\mathcal{B}$ to be fixed and do not address the problem of selecting an optimal orthonormal basis for a specific application. Our goal is to establish properties of the graph Fourier transform that remain valid independent of the choice of the fixed basis. As a motivated example, $\mathcal{B}$ is often taken to be the eigenbasis of the adjacency matrix of a normal graph.

    We may write a signal $f$ either as $(f_1,\dots,f_N)$ or as $(f(1),\dots,f(N))$, depending on the context. Recall that the space $X^N$ is itself a Banach space when equipped with any of the standard $L^p$ norms:
    \begin{equation*}
        \|f\|_p:=\left(\sum_{n=1}^{N}\|f(n)\|^p\right)^{1/p},
    \end{equation*}
    for $p\in[1,\infty)$, and
    \begin{equation*}
        \|f\|_\infty:=\max_{1\leq n\leq N}\|f(n)\|.
    \end{equation*}

    If we endow $V$ with the discrete $\sigma$-algebra and the counting measure $\mu$, we obtain a measure space. Any $X$-valued signal is then a simple, and so strongly measurable, function. Thus, for all $p\in[1,\infty]$, we can also write $X^N=L^p(V,X)$ as the standard Bochner space. Integration with respect to the counting measure takes the form
    \begin{equation*}
        \int_Vfd\mu=\sum_{n=1}^{N}f(n),
    \end{equation*}
    or, using the notation $dn$ for $d\mu$,
    \begin{equation*}
        \int_Vf(n)dn=\sum_{n=1}^{N}f(n), \quad \text{and} \quad \int_Sf(n)dn=\sum_{n\in S}f(n)
    \end{equation*}
    for any $S\subseteq V$.

    Furthermore, since $V$ is finite, all $L^p$ norms are equivalent, as Hölder’s inequality yields
    \begin{equation*}
        \|f\|_q\leq\|f\|_p\leq N^{\frac{1}{p}-\frac{1}{q}}\|f\|_q,
    \end{equation*}
    for all $f:V\to X$, and $1\leq p<q\leq\infty$. Furthermore, both inequalities are sharp, since the first inequality is reached for $f=(x_0,0,\dots,0)$, and the second inequality is reached for $f=(x_0,\dots,x_0)$, with $x_0\in X$ such that $\|x_0\|=1$.

    If $X$ is a Hilbert space, then $X^N$ is a Hilbert space with the inner product
    \begin{equation*}
        \langle f,g\rangle:=\sum_{n=1}^{N}\langle f(n),g(n)\rangle,
    \end{equation*}
    and the induced norm is the $L^2$ norm.
    
    Let $U$ be the unitary matrix whose columns are the basis vectors $u_1,\dots,u_N$, and satisfying $UU^*=I=U^*U$ (where $U$ is an orthogonal matrix if $\mathbb{K}=\mathbb{R}$). We write its generic element as $u_k(n)$, so that $k$ is the column index, and $n$ is the row index. The $k$-th column and the $n$-th row of $U$ are the vectors $u_k$ and $u(n)$, i.e.,
    \begin{equation*}
        u_k=(u_k(1),\dots,u_k(N)), \quad u(n)=(u_1(n),\dots,u_N(n)).
    \end{equation*}
    Consequently, we have the following useful orthonormality relations:
    \begin{equation*}
        \langle u_h,u_k\rangle=\sum_{n=1}^{N}u_h(n)\overline{u_k(n)}=\delta_{hk}, \quad \langle u(n),u(m)\rangle=\sum_{k=1}^{N}u_k(n)\overline{u_k(m)}=\delta_{nm}.
    \end{equation*}
    
    An important parameter associated with the basis $\mathcal{B}$ is the \textit{mutual coherence} (or \textit{coherence}) of $\mathcal{B}$ (\cite{Donoho, Shuman_2, Shuman_1}), defined as the entry-wise $\ell^\infty$-norm of the matrix $U$:
    \begin{equation*}
        \kappa(U):=\|U\|_\infty=\max_{1\leq k,n\leq N}|u_k(n)|.
    \end{equation*}
    More generally, if $p\in[1,\infty]$, we define the \textit{$p$-mutual coherence} (or \textit{$p$-coherence}) of $\mathcal{B}$ as the $\ell^{p,\infty}$-norm of the matrix $U$, i.e.,
    \begin{equation*}
        \kappa_p(U):=\|U\|_{p,\infty}=\max_{1\leq k\leq N}\|u_k\|_p=\max_{1\leq k\leq N}\left(\sum_{n=1}^{N}|u_k(n)|^p\right)^{1/p}.
    \end{equation*}
    This corresponds to the standard norm of the mixed norm space $\ell^{p,\infty}(V\times V)$, which is a Banach space, as discussed in the work of \cite{Benedek}. In particular, note that
    \begin{equation*}
        \kappa_{\infty}(U)=\kappa(U).
    \end{equation*}
    This value measures the spread of the vectors in $\mathcal{B}$ and is a crucial value to consider for estimates involving the operator norm of the graph Fourier operator. Moreover, since $\mathcal{B}$ is an orthonormal basis, $\kappa_p(U)$ satisfies the following bounds:
    \begin{equation*}
        1\leq\kappa_p(U)\leq N^{\frac{1}{p}-\frac{1}{2}} \quad \text{for } 1\leq p<2, \quad \text{and} \quad N^{\frac{1}{p}-\frac{1}{2}}\leq\kappa_p(U)\leq1 \quad \text{for } p>2.
    \end{equation*}
    Indeed, for $1\leq p<2$, by Hölder’s inequality, for all $k\in\lbrace1,\dots,N\rbrace$,
    \begin{equation*}
        1=\|u_k\|_2\leq\|u_k\|_p\leq N^{\frac{1}{p}-\frac{1}{2}}\|u_k\|_2=N^{\frac{1}{p}-\frac{1}{2}}.
    \end{equation*}
    Hence, we get the desired inequality. Analogously, for $p>2$, by Hölder’s inequality, for all $k\in\lbrace1,\dots,N\rbrace$,
    \begin{equation*}
        N^{\frac{1}{p}-\frac{1}{2}}=N^{\frac{1}{p}-\frac{1}{2}}\|u_k\|_2\leq\|u_k\|_p\leq\|u_k\|_2=1.
    \end{equation*}
    Hence, we get the desired inequality. Note that all these inequalities are sharp. The equality with the coefficient $1$ is attained by the identity matrix $I$, and the equality with the coefficient $N^{1/p-1/2}$ is attained by any matrix $U$ such that $|u_k(n)|=N^{-1/2}$ for all $k,n\in\lbrace1,\dots,N\rbrace$. In particular, we have the following:
    \begin{equation*}
        N^{-1/2}\leq\kappa(U)\leq1, \quad \text{and} \quad \kappa_2(U)=1.
    \end{equation*}
    Following \cite{Benedek}, the general $\ell^{p,q}$-norm of the matrix $U$ is given by
    \begin{equation*}
        \|U\|_{p,q}:=\left(\sum_{k=1}^{N}\left(\sum_{n=1}^{N}|u_k(n)|^p\right)^{q/p}\right)^{1/q},
    \end{equation*}
    with obvious modifications when either $p=\infty$ or $q=\infty$.

    We denote the \textit{unitary group} of order $N$ by $\mathrm{U}(N,\mathbb{C})$.

    \subsection{The graph Fourier transform}
    The \textit{graph Fourier transform} and the \textit{inverse graph Fourier transform} of a vector-valued signal $f\in X^N$ are the signals $\hat{f}\in X^N$ and $\check{f}\in X^N$, defined as
	\begin{equation*}
		\hat{f}(k)=\sum_{n=1}^{N}\overline{u_k(n)}f(n), \quad \check{f}(n)=\sum_{k=1}^{N}u_k(n)f(k).
	\end{equation*}
	We may write the graph Fourier transform and the inverse graph Fourier transform as the following linear operators:
	\begin{equation*}
		\mathcal{F}:X^N\to X^N,f\mapsto\mathcal{F}f:=\hat{f}, \quad \text{and} \quad \mathcal{F}^{-1}:X^N\to X^N,f\mapsto\mathcal{F}^{-1}f:=\check{f}.
	\end{equation*}
    In matrix form, these transforms are expressed as $\hat{f}=U^{*}f$ and $\check{f}=Uf$. Note that $\mathcal{F}^{-1}$ is actually the inverse of $\mathcal{F}$. Indeed, by the orthonormality relations,
	\begin{align*}
		\hat{\check{f}}(k)&=\sum_{n=1}^{N}\overline{u_k(n)}\check{f}(n)=\sum_{n=1}^{N}\sum_{h=1}^{N}\overline{u_k(n)}u_h(n)f(h)=\sum_{h=1}^{N}\sum_{n=1}^{N}u_h(n)\overline{u_k(n)}f(h) \\
        &=\sum_{h=1}^{N}\delta_{hk}f(h)=f(k),
	\end{align*}
	and
	\begin{align*}
		\check{\hat{f}}(n)&=\sum_{k=1}^{N}u_k(n)\hat{f}(k)=\sum_{k=1}^{N}\sum_{m=1}^{N}u_k(n)\overline{u_k(m)}f(m)=\sum_{m=1}^{N}\sum_{k=1}^{N}u_k(n)\overline{u_k(m)}f(m) \\
        &=\sum_{m=1}^{N}\delta_{nm}f(m)=f(n).
	\end{align*}
	As a consequence, we have the \textit{inversion formula} for the graph Fourier transform:
	\begin{equation*}
		f(n)=\sum_{k=1}^{N}u_k(n)\hat{f}(k).
	\end{equation*}
	Via the integral formulation introduced, we may write
	\begin{equation*}
		\hat{f}(k)=\int_{V}\overline{u_k(n)}f(n)dn, \quad \text{and} \quad f(n)=\int_{V}u_k(n)\hat{f}(k)dk.
	\end{equation*}

    The definition of graph Fourier transform on an arbitrary (normal) graph derives from the analogous operator defined on the circular graph, with vertex set $V=\mathbb{Z}_N$ and edge set $E=\lbrace([k]_N,[k+1]_N) \mid k\in\mathbb{Z}\rbrace$. So, the adjacency matrix of $G=(V,E)$ is given by
    \begin{equation*}
        a_{hk}=
        \begin{cases}
            1 & \text{if } [k]_N=[h+1]_N, \\
            0 & \text{if } [k]_N\neq[h+1]_N.
        \end{cases}
    \end{equation*}
    For example, if $N=4$, then we have
    \begin{equation*}
        A=
        \begin{pmatrix}
			0 & 1 & 0 & 0 \\
			0 & 0 & 1 & 0 \\
            0 & 0 & 0 & 1 \\
            1 & 0 & 0 & 0 \\
		\end{pmatrix}.
    \end{equation*}
    It is straightforward to see that $A$ is a unitary matrix, and hence a normal matrix. The orthonormal basis of $\mathbb{C}^N$ formed by the eigenvectors of $A$ is given by
    \begin{equation*}
        u_k(n)=\frac{1}{\sqrt{N}}e^{\frac{2\pi i}{N}(n-1)(k-1)}.
    \end{equation*}
    Notice that the coherence of the basis $\mathcal{B}=\lbrace u_1,\dots,u_N\rbrace$ is $\kappa(U)=N^{-1/2}$, achieving the theoretical minimum for an orthonormal basis, implying that the eigenvectors of the directed circular graph are maximally delocalized (see \cite{Shuman_2, Shuman_1}). The resulting graph Fourier transform of a signal $f\in X^N$ is:
    \begin{equation*}
        \hat{f}(k)=\frac{1}{\sqrt{N}}\sum_{n=1}^{N}e^{-\frac{2\pi i}{N}(n-1)(k-1)}f(n),
    \end{equation*}
    which coincides exactly with the Fourier transform on the cyclic group $\mathbb{Z}_N$, i.e., the normalized discrete Fourier transform.
	
	\vspace{0.3cm}
	\textbf{Example:} Let $X=\mathcal{C}([0,1])$ be the Banach space of continuous functions defined on the interval $[0,1]$ equipped with the $\infty$-norm and consider the (directed) circular graph on $N=2$ vertices and the orthonormal basis of $\mathbb{C}^2$ given by the eigenvectors of its adjacency matrix:
	\begin{equation*}
		u_k(n)=\frac{1}{\sqrt{2}}e^{\pi i(n-1)(k-1)}.
	\end{equation*}
	The resulting unitary matrix is
	\begin{equation*}
        U=\frac{1}{\sqrt{2}}
        \begin{pmatrix*}[r]
			1 & 1 \\
			1 & -1 \\
		\end{pmatrix*}.
    \end{equation*}
	The graph Fourier transform is the operator $\mathcal{F}:\mathcal{C}([0,1],\mathbb{C}^2)\to\mathcal{C}([0,1],\mathbb{C}^2)$, given by
	\begin{equation*}
		\mathcal{F}(f,g):=\frac{1}{\sqrt{2}}(f+g,f-g).
	\end{equation*}

    \subsection{Estimates of the operator norms}
    We can establish several fundamental bounds for the operator norm of the graph Fourier transform. These estimates provide a quantitative measure of the spectral spread of signals and serve as a prerequisite for uncertainty principles on graphs.
    
    \begin{Theorem} \label{Continuity of Fourier transform 1}
        Fix an orthonormal basis $\mathcal{B}=\lbrace u_1,\dots,u_N\rbrace$ of $\mathbb{K}^N$, a Banach space $X$, and $f\in X^N$. Then, for $p\in[1,\infty]$,
        \begin{equation*}
            \|\hat{f}\|_\infty\leq\kappa_{p'}(U)\|f\|_p, \quad \|f\|_\infty\leq\kappa_{p'}(U^*)\|\hat{f}\|_p,
        \end{equation*}
        where $p'$ is the conjugate exponent of $p$. In particular, the following estimates hold:
        \begin{equation*}
            \|\hat{f}\|_\infty\leq\|U\|_\infty\|f\|_1, \quad \|f\|_\infty\leq\|U\|_\infty\|\hat{f}\|_1, \quad \|\hat{f}\|_\infty\leq\|f\|_2, \quad \|f\|_\infty\leq\|\hat{f}\|_2.
        \end{equation*}
		Consequently, $\mathcal{F}:L^p(V,X)\to L^\infty(V,X)$ is a topological isomorphism, and
        \begin{equation*}
            \|\mathcal{F}\|_{p\to\infty}=\kappa_{p'}(U).
        \end{equation*}
    \end{Theorem}
	\begin{proof}
		For $p\in[1,\infty]$, by Hölder's inequality,
        \begin{equation*}
            \|\hat{f}(k)\|=\left\|\sum_{n=1}^{N}\overline{u_k(n)}f(n)\right\|\leq\sum_{n=1}^{N}|u_k(n)|\|f(n)\|\leq\|u_k\|_{p'}\|f\|_p,
        \end{equation*}
        for all $k\in\lbrace1,\dots,N\rbrace$, and
        \begin{equation*}
            \|f(n)\|=\left\|\sum_{k=1}^{N}u_k(n)\hat{f}(k)\right\|\leq\sum_{k=1}^{N}|u_k(n)|\|\hat{f}(k)\|\leq\|u(n)\|_{p'}\|\hat{f}\|_p,
        \end{equation*}
        for all $n\in\lbrace1,\dots,N\rbrace$. Consequently, we get
        \begin{equation*}
            \|\hat{f}\|_\infty\leq\kappa_{p'}(U)\|f\|_p, \quad \|f\|_\infty\leq\kappa_{p'}(U^*)\|\hat{f}\|_p.
        \end{equation*}
        In particular, the following estimates hold:
        \begin{equation*}
            \|\hat{f}\|_\infty\leq\|U\|_\infty\|f\|_1, \quad \|f\|_\infty\leq\|U\|_\infty\|\hat{f}\|_1, \quad \|\hat{f}\|_\infty\leq\|f\|_2, \quad \|f\|_\infty\leq\|\hat{f}\|_2.
        \end{equation*}
        These estimates show that $\|\mathcal{F}\|_{p\to\infty}\leq\kappa_{p'}(U)$. We need the reverse inequality. Fix $x_0\in X$, with $\|x_0\|=1$, and distinguish three cases.
        
        \noindent \textit{1° case}: Let $p=1$, so that $p'=\infty$, and $k^*,n^*\in\lbrace1,\dots,N\rbrace$ such that
        \begin{equation*}
            \kappa(U)=\max_{1\leq k\leq N}\|u_k\|_\infty=|u_{k^*}(n^*)|.
        \end{equation*}
        Now consider the signal $f:=(0,\dots,x_0,\dots,0)$ having $x_0$ in the $n^*$-th position, such that $\|f\|_1=\|x_0\|=1$. Thus,
        \begin{equation*}
            \hat{f}(k^*)=\sum_{n=1}^{N}\overline{u_{k^*}(n)}f(n)=\overline{u_{k^*}(n^*)}x_0,
        \end{equation*}
        and so
        \begin{equation*}
            \frac{\|\hat{f}\|_\infty}{\|f\|_1}\geq\|\hat{f}(k^*)\|=|u_{k^*}(n^*)|=\kappa(U).
        \end{equation*}
        Hence, we get the desired inequality.
        
        \noindent \textit{2° case}: Let $1<p<\infty$ and $k^*\in\lbrace1,\dots,N\rbrace$ such that
        \begin{equation*}
            \kappa_{p'}(U)=\max_{1\leq k\leq N}\|u_k\|_{p'}=\|u_{k^*}\|_{p'}.
        \end{equation*}
        Now consider the signal $f(n):=\sgn(u_{k^*}(n))|u_{k^*}(n)|^{p'/p}x_0$, where
        \begin{equation*}
            \sgn(z):=
            \begin{cases}
                z/|z| & \text{if } z\neq0, \\
                0 & \text{if } z=0,
            \end{cases}
        \end{equation*}
        is the sign function of complex numbers, so that $\overline{z}\sgn(z)=|z|$, for all $z\in\mathbb{C}$. Thus,
        \begin{equation*}
            \|f\|_p=\left(\sum_{n=1}^{N}\|f(n)\|^p\right)^{1/p}=\left(\sum_{n=1}^{N}|u_{k^*}(n)|^{p'}\right)^{1/p}=\|u_{k^*}\|_{p'}^{p'/p},
        \end{equation*}
        and, by using that $1+p'/p=p'$,
        \begin{equation*}
            \hat{f}(k^*)=\sum_{n=1}^{N}\overline{u_{k^*}(n)}\sgn(u_{k^*}(n))|u_{k^*}(n)|^{p'/p}x_0=\sum_{n=1}^{N}|u_{k^*}(n)|^{1+p'/p}x_0=\|u_{k^*}\|_{p'}^{p'}x_0.
        \end{equation*}
        So,
        \begin{equation*}
            \frac{\|\hat{f}\|_\infty}{\|f\|_p}\geq\frac{\|\hat{f}(k^*)\|}{\|f\|_p}=\frac{\|u_{k^*}\|_{p'}^{p'}}{\|u_{k^*}\|_{p'}^{p'/p}}=\|u_{k^*}\|_{p'}^{p'(1-1/p)}=\|u_{k^*}\|_{p'}=\kappa_{p'}(U).
        \end{equation*}
        Hence, we get the desired inequality.
        
        \noindent \textit{3° case}: Let $p=\infty$, so that $p'=1$, and $k^*\in\lbrace1,\dots,N\rbrace$ such that
        \begin{equation*}
            \kappa_1(U)=\max_{1\leq k\leq N}\|u_k\|_1=\|u_{k^*}\|_1.
        \end{equation*}
        Now consider the signal $f(n):=\sgn(u_{k^*}(n))x_0$, note that $\|f\|_\infty=\|x_0\|=1$ and
        \begin{equation*}
            \hat{f}(k^*)=\sum_{n=1}^{N}\overline{u_{k^*}(n)}\sgn(u_{k^*}(n))x_0=\sum_{n=1}^{N}|u_{k^*}(n)|x_0=\|u_{k^*}\|_1x_0.
        \end{equation*}
        So,
        \begin{equation*}
            \frac{\|\hat{f}\|_\infty}{\|f\|_1}\geq\|\hat{f}(k^*)\|=\|u_{k^*}\|_1=\kappa_1(U).
        \end{equation*}
        Hence, we get the desired inequality. \qedhere
	\end{proof}
    
    \begin{Theorem} \label{Continuity of Fourier transform 2}
        Fix an orthonormal basis $\mathcal{B}=\lbrace u_1,\dots,u_N\rbrace$ of $\mathbb{K}^N$, a Banach space $X$, and $f\in X^N$. Then, for $q\in[1,\infty]$,
        \begin{equation*}
            \|\hat{f}\|_q\leq\kappa_q(U^*)\|f\|_1, \quad \|f\|_q\leq\kappa_q(U)\|\hat{f}\|_1.
        \end{equation*}
        In particular, the following estimates hold:
        \begin{equation*}
            \|\hat{f}\|_\infty\leq\|U\|_\infty\|f\|_1, \quad \|f\|_\infty\leq\|U\|_\infty\|\hat{f}\|_1, \quad \|\hat{f}\|_2\leq\|f\|_1, \quad \|f\|_2\leq\|\hat{f}\|_1.
        \end{equation*}
		Consequently, $\mathcal{F}:L^1(V,X)\to L^q(V,X)$ is a topological isomorphism, and
        \begin{equation*}
            \|\mathcal{F}\|_{1\to q}=\kappa_q(U^*).
        \end{equation*}
    \end{Theorem}
    \begin{proof}
        For $q\in[1,\infty]$, by Minkowski's integral inequality (see \cite{Folland}, Theorem 6.19, p. 194),
        \begin{align*}
            \|\hat{f}\|_q&=\left(\sum_{k=1}^{N}\|\hat{f}(k)\|^q\right)^{1/q}=\left(\sum_{k=1}^{N}\left\|\sum_{n=1}^{N}\overline{u_k(n)}f(n)\right\|^q\right)^{1/q} \\
            &\leq\sum_{n=1}^{N}\left(\sum_{k=1}^{N}|u_k(n)|^q\|f(n)\|^q\right)^{1/q}=\sum_{n=1}^{N}\|u(n)\|_q\|f(n)\|\leq\kappa_q(U^*)\|f\|_1,
        \end{align*}
        and
        \begin{align*}
            \|f\|_q&=\left(\sum_{n=1}^{N}\|f(n)\|^q\right)^{1/q}=\left(\sum_{n=1}^{N}\left\|\sum_{k=1}^{N}u_k(n)\hat{f}(k)\right\|^q\right)^{1/q} \\
            &\leq\sum_{k=1}^{N}\left(\sum_{n=1}^{N}|u_k(n)|^q\|\hat{f}(k)\|^q\right)^{1/q}=\sum_{k=1}^{N}\|u_k\|_q\|\hat{f}(k)\|\leq\kappa_q(U)\|\hat{f}\|_1.
        \end{align*}
        In particular, the following estimates hold:
        \begin{equation*}
            \|\hat{f}\|_\infty\leq\|U\|_\infty\|f\|_1, \quad \|f\|_\infty\leq\|U\|_\infty\|\hat{f}\|_1, \quad \|\hat{f}\|_2\leq\|f\|_1, \quad \|f\|_2\leq\|\hat{f}\|_1.
        \end{equation*}
        These estimates show that $\|\mathcal{F}\|_{1\to q}\leq\kappa_q(U^*)$. We need the reverse inequality. Fix $x_0\in X$, with $\|x_0\|=1$, and let $n^*\in\lbrace1,\dots,N\rbrace$ such that
        \begin{equation*}
            \kappa_q(U^*)=\max_{1\leq n\leq N}\|u(n)\|_q=\|u(n^*)\|_q.
        \end{equation*}
        Now consider the signal $f:=(0,\dots,x_0,\dots,0)$ having $x_0$ in the $n^*$-th position, such that $\|f\|_1=\|x_0\|=1$. Thus,
        \begin{equation*}
            \hat{f}(k)=\sum_{n=1}^{N}\overline{u_k(n)}f(n)=\overline{u_k(n^*)}x_0,
        \end{equation*}
        and so, if $1\leq q<\infty$,
        \begin{equation*}
            \frac{\|\hat{f}\|_q}{\|f\|_1}=\left(\sum_{k=1}^{N}\|\hat{f}(k)\|^q\right)^{1/q}=\left(\sum_{k=1}^{N}|u_k(n^*)|^q\right)^{1/q}=\|u(n^*)\|_q=\kappa_q(U^*).
        \end{equation*}
        Hence, we get the desired inequality. If $q=\infty$, we obtain the inequality by Theorem \ref{Continuity of Fourier transform 1}. \qedhere
    \end{proof}
    
    \begin{Theorem} \label{Continuity of Fourier transform 3}
        Fix an orthonormal basis $\mathcal{B}=\lbrace u_1,\dots,u_N\rbrace$ of $\mathbb{K}^N$, a Banach space $X$, and $f\in X^N$. Then, for $p,q\in[1,\infty]$,
        \begin{equation*}
            \|\hat{f}\|_q\leq\|U\|_{p',q}\|f\|_p, \quad \|f\|_q\leq\|U^*\|_{p',q}\|\hat{f}\|_p,
        \end{equation*}
        where $p'$ is the conjugate exponent of $p$. In particular, the following estimates hold:
        \begin{equation*}
            \|\hat{f}\|_2\leq\sqrt{N}\|f\|_2, \quad \|f\|_2\leq\sqrt{N}\|\hat{f}\|_2.
        \end{equation*}
		Consequently, $\mathcal{F}:L^p(V,X)\to L^q(V,X)$ is a topological isomorphism, and
		\begin{equation*}
			\|\mathcal{F}\|_{p\to q}\leq\|U\|_{p',q}.
		\end{equation*}
    \end{Theorem}
	\begin{proof}
        For $p,q\in[1,\infty]$, by Hölder's inequality,
        \begin{equation*}
            \|\hat{f}(k)\|=\left\|\sum_{n=1}^{N}\overline{u_k(n)}f(n)\right\|\leq\sum_{n=1}^{N}|u_k(n)|\|f(n)\|\leq\|u_k\|_{p'}\|f\|_p,
        \end{equation*}
        for all $k\in\lbrace1,\dots,N\rbrace$, and
        \begin{equation*}
            \|f(n)\|=\left\|\sum_{k=1}^{N}u_k(n)\hat{f}(k)\right\|\leq\sum_{k=1}^{N}|u_k(n)|\|\hat{f}(k)\|\leq\|u(n)\|_{p'}\|\hat{f}\|_p,
        \end{equation*}
        for all $n\in\lbrace1,\dots,N\rbrace$. Consequently, we get
        \begin{equation*}
            \|\hat{f}\|_q=\left(\sum_{k=1}^{N}\|\hat{f}(k)\|^q\right)^{1/q}\leq\left(\sum_{k=1}^{N}\|u_k\|_{p'}^q\|f\|_p^q\right)^{1/q}=\|U\|_{p',q}\|f\|_p,
        \end{equation*}
        and
        \begin{equation*}
            \|f\|_q=\left(\sum_{n=1}^{N}\|f(n)\|^q\right)^{1/q}\leq\left(\sum_{n=1}^{N}\|u(n)\|_{p'}^q\|\hat{f}\|_p^q\right)^{1/q}=\|U^*\|_{p',q}\|\hat{f}\|_p.
        \end{equation*}
        In particular, the following estimates hold:
        \begin{equation*}
            \|\hat{f}\|_2\leq\sqrt{N}\|f\|_2, \quad \|f\|_2\leq\sqrt{N}\|\hat{f}\|_2.
        \end{equation*}
        These estimates show that $\|\mathcal{F}\|_{p\to q}\leq\|U\|_{p',q}$. \qedhere
	\end{proof}
	
	\begin{Theorem} \label{Parseval and Plancherel}
        Fix an orthonormal basis $\mathcal{B}=\lbrace u_1,\dots,u_N\rbrace$ of $\mathbb{K}^N$, a Hilbert space $X$, and $f,g\in X^N$. Then,
        \begin{alignat*}{4}
            \langle\hat{f},\hat{g}\rangle&=\langle f,g\rangle &&\qquad (\text{Parseval's identity}), \\
            \|\hat{f}\|_2&=\|f\|_2 &&\qquad (\text{Plancherel's equality}).
        \end{alignat*}
        Consequently, $\mathcal{F}:L^2(V,X)\to L^2(V,X)$ is an isometric isomorphism (equivalently, $\mathcal{F}$ is a unitary operator), and
        \begin{equation*}
			\|\mathcal{F}\|_{2\to 2}=1.
		\end{equation*}
    \end{Theorem}
	\begin{proof}
        This is a straightforward calculation:
        \begin{align*}
			\langle\hat{f},\hat{g}\rangle&=\sum_{k=1}^{N}\langle\hat{f}(k),\hat{g}(k)\rangle=\sum_{k=1}^{N}\left\langle\sum_{n=1}^{N}\overline{u_k(n)}f(n),\hat{g}(k)\right\rangle=\sum_{k=1}^{N}\sum_{n=1}^{N}\overline{u_k(n)}\langle f(n),\hat{g}(k)\rangle \\
            &=\sum_{n=1}^{N}\sum_{k=1}^{N}\langle f(n),u_k(n)\hat{g}(k)\rangle=\sum_{n=1}^{N}\left\langle f(n),\sum_{k=1}^{N}u_k(n)\hat{g}(k)\right\rangle=\sum_{n=1}^{N}\langle f(n),g(n)\rangle \\
            &=\langle f,g\rangle.
		\end{align*}
        This equality shows that $\mathcal{F}$ is a unitary operator, and $\|\mathcal{F}\|_{2\to 2}=1$. \qedhere
	\end{proof}
	
	\textbf{Remark:} Theorem \ref{Parseval and Plancherel} shows the validity of Parseval's identity and Plancherel's equality for $\mathcal{F}:L^2(V,X)\to L^2(V,X)$, if $X$ is a Hilbert space, and yields the sharp value $\|\mathcal{F}\|_{2\to 2}=1$. The operator norms for the Fourier transform $\mathcal{F}$ across different $L^p$ spaces are summarized in Table \ref{Fourier_norms}.
    \begin{table}[ht]
        \centering
        \caption{Operator norms for the vector-valued Fourier transform $\mathcal{F}$}
        \label{Fourier_norms}
        \renewcommand{\arraystretch}{1.5} 
        \begin{tabular}{ccl} 
            \toprule
            Operator & Operator norm & Note \\ \midrule
            $\mathcal{F}:L^p(V,X)\to L^\infty(V,X)$ & $=\kappa_{p'}(U)$   & - \\
            $\mathcal{F}:L^1(V,X)\to L^q(V,X)$      & $=\kappa_q(U^*)$    & - \\
            $\mathcal{F}:L^p(V,X)\to L^q(V,X)$      & $\leq\|U\|_{p',q}$  & - \\
            $\mathcal{F}:L^2(V,X)\to L^2(V,X)$      & $=1$                & If $X$ is a Hilbert space \\
            \bottomrule
            \end{tabular}
    \end{table}
    
	\textbf{Remark:} Theorems \ref{Continuity of Fourier transform 1} and \ref{Continuity of Fourier transform 2} provide the sharp values of the operator norms $\|\mathcal{F}\|_{p\to\infty}$ and $\|\mathcal{F}\|_{1\to q}$, involving the $p'$-coherence of $U$ and the $q$-coherence of $U^*$, respectively. In contrast, the estimate given in Theorem \ref{Continuity of Fourier transform 3} is not sharp, in general, and depends heavily on the matrix structure. For example, if $U=I$ is the identity matrix and  $X$ is a Hilbert space, Theorem \ref{Parseval and Plancherel} yields:
    \begin{equation*}
        \|\mathcal{F}\|_{2\to 2}=1<\sqrt{N}=\|I\|_{2,2}.
    \end{equation*}
    As a further counterexample, let $X=\mathbb{C}^2$ be the Banach space of complex vectors, equipped with the $\infty$-norm, and consider the (directed) circular graph with $N=2$ vertices and the orthonormal basis $\mathcal{B}$ of $\mathbb{C}^2$ given by the eigenvectors of its adjacency matrix:
	\begin{equation*}
		u_k(n)=\frac{1}{\sqrt{2}}e^{\pi i(n-1)(k-1)}.
	\end{equation*}
	The resulting unitary matrix is
	\begin{equation*}
        U=\frac{1}{\sqrt{2}}
        \begin{pmatrix*}[r]
			1 & 1 \\
			1 & -1 \\
		\end{pmatrix*}.
    \end{equation*}
	The graph Fourier transform is the operator $\mathcal{F}:\mathbb{C}^2\times\mathbb{C}^2\to\mathbb{C}^2\times\mathbb{C}^2$, given by
	\begin{equation*}
		\mathcal{F}(x,y):=\frac{1}{\sqrt{2}}(x+y,x-y).
	\end{equation*}
	For the vectors $x=(1,1)$ and $y=(1,-1)$, we have $\|x\|_\infty=1$ and $\|y\|_\infty=1$, such that
	\begin{equation*}
		\|(x,y)\|_2=\sqrt{1+1}=\sqrt{2}.
	\end{equation*}
	However,
	\begin{equation*}
		\mathcal{F}(x,y)=\frac{1}{\sqrt{2}}((2,0),(0,2)), \quad \text{implying} \quad \|\mathcal{F}(x,y)\|_2=\frac{1}{\sqrt{2}}\sqrt{4+4}=2.
	\end{equation*}
	This shows that $\mathcal{F}:L^2(V,X)\to L^2(V,X)$ is not an isometry, as
	\begin{equation*}
		\|\mathcal{F}\|_{2\to 2}\geq\frac{\|\mathcal{F}(x,y)\|_2}{\|(x,y)\|_2}=\sqrt{2}>1.
	\end{equation*}
    In particular, since $\|\mathcal{F}\|_{2\to 2}\leq\sqrt{2}$ by Theorem \ref{Continuity of Fourier transform 3}, we have exactly $\|\mathcal{F}\|_{2\to 2}=\sqrt{2}$. The following theorem provides the general bound applicable to any Banach space.
    
	\begin{Theorem} \label{Characterization of Plancherel}
        Fix a Banach space $X$, and $N\in\mathbb{N}$.
        \begin{enumerate}[label=(\alph*)]
            \item Fix an orthonormal basis $\mathcal{B}=\lbrace u_1,\dots,u_N\rbrace$ of $\mathbb{K}^N$. Then, $\mathcal{F}:L^2(V,X)\to L^2(V,X)$ is a topological isomorphism, and $1\leq\|\mathcal{F}\|_{2\to 2}\leq N^{1/2}$.
            \item If $N=1$, then $\|\hat{f}\|_2=\|f\|_2$ for all $u\in\mathbb{S}^1$ and $f\in X$. So, $\mathcal{F}:L^2(V,X)\to L^2(V,X)$ is an isometric isomorphism, and $\|\mathcal{F}\|_{2\to 2}=1$.
            \item If $N\geq2$, then $\|\hat{f}\|_2=\|f\|_2$ for all $U\in\mathrm{U}(N,\mathbb{C})$ and $f\in X^N$ if and only if $X$ is a Hilbert space.
            \item If $N\geq2$, then $\|\mathcal{F}\|_{2\to 2}=1$ for all $U\in\mathrm{U}(N,\mathbb{C})$ if and only if $X$ is a Hilbert space.
        \end{enumerate}
    \end{Theorem}
    \begin{proof}
        We prove the four properties separately.
        \begin{enumerate}[label=(\alph*), leftmargin=*, widest=a, font=\bfseries]
            \item Fix $x_0\in X$, with $\|x_0\|=1$, and a scalar signal $\alpha\in\mathbb{K}^N$. Hence, by Plancherel's equality on the Hilbert space $\mathbb{K}^N$,
            \begin{equation*}
                \sum_{k=1}^{N}|\hat{\alpha}(k)|^2=\|\hat{\alpha}\|_2^2=\|\alpha\|_2^2=\sum_{n=1}^{N}|\alpha(n)|^2.
            \end{equation*}
            Now consider the signal $f(n):=\alpha(n)x_0$, so that
            \begin{align*}
                \|\hat{f}\|_2^2&=\sum_{k=1}^{N}\|\hat{f}(k)\|^2=\sum_{k=1}^{N}\left\|\sum_{n=1}^{N}\overline{u_k(n)}f(n)\right\|^2=\sum_{k=1}^{N}\left\|\sum_{n=1}^{N}\overline{u_k(n)}\alpha(n)x_0\right\|^2 \\
                &=\sum_{k=1}^{N}\left|\sum_{n=1}^{N}\overline{u_k(n)}\alpha(n)\right|^2=\sum_{k=1}^{N}|\hat{\alpha}(k)|^2=\sum_{n=1}^{N}|\alpha(n)|^2=\sum_{n=1}^{N}\|\alpha(n)x_0\|^2 \\
                &=\sum_{n=1}^{N}\|f(n)\|^2=\|f\|_2^2.
            \end{align*}
            So, we have $\|\mathcal{F}\|_{2\to 2}\geq\|\hat{f}\|_2/\|f\|_2=1$. The inequality $\|\mathcal{F}\|_{2\to 2}\leq N^{1/2}$ is given by Theorem \ref{Continuity of Fourier transform 3}.
            \item When $N=1$, the orthonormal basis consists of a single vector $u\in\mathbb{C}$ such that $|u|=1$ (i.e., $u\in\mathbb{S}^1$). For a signal $f\in X$, the Fourier transform is given by the scalar multiplication, $\hat{f}=uf$. Using the definition of the $L^2$-norm for $N=1$, we have:
            \begin{equation*}
                \|\hat{f}\|_2=\|uf\|=|u|\|f\|=\|f\|_2.
            \end{equation*}
            This equality shows that $\mathcal{F}$ is an isometric isomorphism, and $\|\mathcal{F}\|_{2\to 2}=1$.
            \item The left implication holds by Plancherel's equality. So, we assume that $\|\hat{f}\|_2=\|f\|_2$ for all $U\in\mathrm{U}(N,\mathbb{C})$ and $f\in X^N$, and we show that $X$ is a Hilbert space. By Jordan–von Neumann Theorem (\cite{Jordan}), it is sufficient to show that the norm on $X$ satisfies the parallelogram law:
            \begin{equation*}
                \|x+y\|^2+\|x-y\|^2=2\|x\|^2+2\|y\|^2 \quad \text{for all } x,y\in X.
            \end{equation*}
            Thus, let $x,y\in X$ and consider the signal $f:=(x,y,0,\dots,0)\in X^N$, such that $\|f\|_2^2=\|x\|^2+\|y\|^2$. Also, consider the unitary matrix $U\in\mathrm{U}(N,\mathbb{C})$ as follows:
            \begin{equation*}
                U:=\frac{1}{\sqrt{2}}
                \begin{pmatrix*}[r]
                    1 & 1 & \mathbf{0} \\
                    1 & -1 & \mathbf{0} \\
                    \mathbf{0} & \mathbf{0} & I
                \end{pmatrix*}, \quad \text{implying} \quad \hat{f}=\frac{1}{\sqrt{2}}(x+y,x-y,0,\dots,0).
            \end{equation*}
            Hence, the equality $\|\hat{f}\|_2=\|f\|_2$ is exactly the parallelogram law for the norm on $X$. Consequently, $X$ is a Hilbert space.
            \item The left implication holds by Plancherel's equality. We suppose that $\|\mathcal{F}\|_{2\to 2}=1$ for all $U\in\mathrm{U}(N,\mathbb{C})$, and we show that $X$ is a Hilbert space. Fix an arbitrary unitary matrix $U\in\mathrm{U}(N,\mathbb{C})$. Since $U^*$ is also a unitary matrix, we obtain that $\|\mathcal{F}^{-1}\|_{2\to 2}=1$. Hence, for all $f\in X^N$,
            \begin{equation*}
                \|\hat{f}\|_2\leq\|f\|_2, \quad \|f\|_2\leq\|\hat{f}\|_2.
            \end{equation*}
            Thus, $\|\hat{f}\|_2=\|f\|_2$ for all $f\in X^N$ and, by part (c), $X$ is a Hilbert space. \qedhere
        \end{enumerate}
    \end{proof}
    
	\textbf{Remark:} The bounds proved in Theorem \ref{Characterization of Plancherel}(a) are sharp. Indeed, $\|\mathcal{F}\|_{2\to 2}=1$ holds for Hilbert spaces, while the bound $\|\mathcal{F}\|_{2\to 2}=N^{1/2}$ holds for certain  Banach spaces, as shown in the previous example.

    \section{Uncertainty principles on graphs}
    In this section, we discuss the validity of lower bounds for products of $L^p$ norms of signals, which we interpret as uncertainty principles.
    
    The core idea behind uncertainty principles, whether in physics, signal processing, or mathematics, is a fundamental trade-off between the localization of a signal and the localization of its Fourier transform. In harmonic analysis, there are many examples of uncertainty principles, each one stating that a non-zero function $f$ and its Fourier transform $\hat{f}$ cannot both be sharply localized simultaneously. Therefore, if a signal is very short in time, it must contain a wide range of frequencies to build that sharpness.

    In classical time-frequency analysis, a very elementary uncertainty principle can be stated for a signal $f\in L^1(\mathbb{R}^d)$ such that $\hat{f}\in L^1(\mathbb{R}^d)$. In particular, the following holds:
    \begin{equation*}
        \frac{\|f\|_1\|\hat{f}\|_1}{\|f\|_\infty\|\hat{f}\|_\infty}\geq1.
    \end{equation*}
    The ratio $\|f\|_1/\|f\|_\infty$ serves as a proxy for the size of the signal's support (a measure of its non-sparsity).
    
    In the following, we extend this inequality to the graph setting and derive analogous inequalities for graph signals (see \cite{Agaskar,Emmrich,Ricaud,Tsitsvero}).
	
	\begin{Theorem} \label{Uncertainty principle 1}
        Fix an orthonormal basis $\mathcal{B}=\lbrace u_1,\dots,u_N\rbrace$ of $\mathbb{K}^N$, a Banach space $X$, and $f\in X^N-\lbrace0\rbrace$. Then, for $p\in[1,\infty]$,
        \begin{equation*}
            \frac{\|f\|_p\|\hat{f}\|_p}{\|f\|_\infty\|\hat{f}\|_\infty}\geq\frac{1}{\kappa_{p'}(U)\kappa_{p'}(U^*)},
        \end{equation*}
        where $p'$ is the conjugate exponent of $p$. In particular, the following estimates hold:
        \begin{equation*}
            \frac{\|f\|_1\|\hat{f}\|_1}{\|f\|_\infty\|\hat{f}\|_\infty}\geq\frac{1}{\kappa(U)^2}, \qquad \frac{\|f\|_2\|\hat{f}\|_2}{\|f\|_\infty\|\hat{f}\|_\infty}\geq1.
        \end{equation*}
    \end{Theorem}
	\begin{proof}
        By Theorem \ref{Continuity of Fourier transform 1}, we know that
        \begin{equation*}
            \|\hat{f}\|_\infty\leq\kappa_{p'}(U)\|f\|_p, \quad \|f\|_\infty\leq\kappa_{p'}(U^*)\|\hat{f}\|_p.
        \end{equation*}
        Thus, $\|f\|_\infty\|\hat{f}\|_\infty\leq\kappa_{p'}(U)\kappa_{p'}(U^*)\|f\|_p\|\hat{f}\|_p$, and so
        \begin{equation*}
            \frac{\|f\|_p\|\hat{f}\|_p}{\|f\|_\infty\|\hat{f}\|_\infty}\geq\frac{1}{\kappa_{p'}(U)\kappa_{p'}(U^*)}. \qedhere
        \end{equation*}
	\end{proof}
    
	\begin{Theorem} \label{Uncertainty principle 2}
        Fix an orthonormal basis $\mathcal{B}=\lbrace u_1,\dots,u_N\rbrace$ of $\mathbb{K}^N$, a Banach space $X$, and $f\in X^N-\lbrace0\rbrace$. Then, for $q\in[1,\infty]$,
        \begin{equation*}
            \frac{\|f\|_1\|\hat{f}\|_1}{\|f\|_q\|\hat{f}\|_q}\geq\frac{1}{\kappa_q(U)\kappa_q(U^*)}.
        \end{equation*}
        In particular, the following estimates hold:
        \begin{equation*}
            \frac{\|f\|_1\|\hat{f}\|_1}{\|f\|_\infty\|\hat{f}\|_\infty}\geq\frac{1}{\kappa(U)^2}, \qquad \frac{\|f\|_1\|\hat{f}\|_1}{\|f\|_2\|\hat{f}\|_2}\geq1.
        \end{equation*}
    \end{Theorem}
	\begin{proof}
        By Theorem \ref{Continuity of Fourier transform 2}, we know that
        \begin{equation*}
            \|\hat{f}\|_q\leq\kappa_q(U^*)\|f\|_1, \quad \|f\|_q\leq\kappa_q(U)\|\hat{f}\|_1.
        \end{equation*}
        Thus, $\|f\|_q\|\hat{f}\|_q\leq\kappa_q(U)\kappa_q(U^*)\|f\|_1\|\hat{f}\|_1$, and so
        \begin{equation*}
            \frac{\|f\|_1\|\hat{f}\|_1}{\|f\|_q\|\hat{f}\|_q}\geq\frac{1}{\kappa_q(U)\kappa_q(U^*)}. \qedhere
        \end{equation*}
	\end{proof}
    
	\begin{Theorem} \label{Uncertainty principle 3}
        Fix an orthonormal basis $\mathcal{B}=\lbrace u_1,\dots,u_N\rbrace$ of $\mathbb{K}^N$, a Banach space $X$, and $f\in X^N-\lbrace0\rbrace$. Then, for $p,q\in[1,\infty]$,
        \begin{equation*}
            \frac{\|f\|_p\|\hat{f}\|_p}{\|f\|_q\|\hat{f}\|_q}\geq\frac{1}{\|U\|_{p',q}\|U^*\|_{p',q}},
        \end{equation*}
        where $p'$ is the conjugate exponent of $p$.
    \end{Theorem}
	\begin{proof}
        By Theorem \ref{Continuity of Fourier transform 3}, we know that
        \begin{equation*}
            \|\hat{f}\|_q\leq\|U\|_{p',q}\|f\|_p, \quad \|f\|_q\leq\|U^*\|_{p',q}\|\hat{f}\|_p,
        \end{equation*}
        Thus, $\|f\|_q\|\hat{f}\|_q\leq\|U\|_{p',q}\|U^*\|_{p',q}\|f\|_p\|\hat{f}\|_p$, and so
        \begin{equation*}
            \frac{\|f\|_p\|\hat{f}\|_p}{\|f\|_q\|\hat{f}\|_q}\geq\frac{1}{\|U\|_{p',q}\|U^*\|_{p',q}}. \qedhere
        \end{equation*}
	\end{proof}

    \section{Convolution of vector-valued graph signals}
    In this section, we define the convolution operator for vector-valued graph signals via the graph Fourier transform, and investigate its fundamental properties. In the classical setting, Young's inequality establishes the continuity of the convolution as a bilinear operator between $L^p$ spaces. We demonstrate that a similar result holds in the graph setting, with the operator bound depending on the geometry of the graph through the entries of the orthonormal basis.

    \subsection{Basic algebraic properties of convolution}
    In classical time-frequency analysis, the convolution operator is characterized by its relationship with the Fourier transform. In particular, the Fourier transform of the convolution of two signals is the point-wise product of their respective Fourier transforms.

    In order to maintain this fundamental property valid, Shuman et al. \cite{Shuman_2} defined the convolution operator for graph signals as the inverse Fourier transform of the point-wise product of their Fourier transform, giving rise to the following definition, which we extend to the setting of vector-valued graph signals:
	\begin{equation*}
		*:\mathbb{K}^N\times X^N\to X^N,(\alpha,f)\mapsto\alpha*f, \qquad (\alpha*f)(n):=\sum_{k=1}^{N}u_k(n)\hat{\alpha}(k)\hat{f}(k).
	\end{equation*}
    Note that the convolution is an operator between a scalar signal and a vector-valued signal, and so $\alpha*f$ is still a vector-valued signal.
    
    Via the integral formulation introduced, we may write
	\begin{equation*}
		(\alpha*f)(n)=\int_{V}u_k(n)\hat{\alpha}(k)\hat{f}(k)dk.
	\end{equation*}
	
	\begin{Theorem} \label{Properties of convolution}
        Fix an orthonormal basis $\mathcal{B}=\lbrace u_1,\dots,u_N\rbrace$ of $\mathbb{K}^N$, a Banach space $X$, $\alpha,\beta\in\mathbb{K}^N$, and $f\in X^N$.
        \begin{enumerate}[label=(\alph*)]
            \item The map $*:\mathbb{K}^N\times X^N\to X^N,(\alpha,x)\mapsto\alpha*x$ is a bilinear operator.
            \item $\widehat{\alpha*f}=\hat{\alpha}\hat{f}$.
            \item $\alpha*\beta=\beta*\alpha$.
            \item $(\alpha*\beta)*f=\alpha*(\beta*f)$.
            \item If $\varepsilon:=u_1+\cdots+u_N$, then $\varepsilon*f=f$.
        \end{enumerate}
    \end{Theorem}
	\begin{proof}
        We prove the five properties separately.
        \begin{enumerate}[label=(\alph*), leftmargin=*, widest=a, font=\bfseries]
            \item For all $n\in\lbrace1,\dots,N\rbrace$,
            \begin{align*}
				(\alpha*(\lambda f+\mu g))(n)&=\sum_{k=1}^{N}u_k(n)\hat{\alpha}(k)\widehat{\lambda f+\mu g}(k)=\sum_{k=1}^{N}u_k(n)\hat{\alpha}(k)(\lambda\hat{f}(k)+\mu\hat{g}(k)) \\
				&=\lambda\sum_{k=1}^{N}u_k(n)\hat{\alpha}(k)\hat{f}(k)+\mu\sum_{k=1}^{N}u_k(n)\hat{\alpha}(k)\hat{g}(k) \\
				&=\lambda(\alpha*f)(n)+\mu(\alpha*g)(n)=(\lambda(\alpha*f)+\mu(\alpha*g))(n),
			\end{align*}
			and
			\begin{align*}
				((\lambda\alpha+\mu\beta)*f)(n)&=\sum_{k=1}^{N}u_k(n)\widehat{\lambda\alpha+\mu\beta}(k)\hat{f}(k)=\sum_{k=1}^{N}u_k(n)(\lambda\hat{\alpha}(k)+\mu\hat{\beta}(k))\hat{f}(k) \\
				&=\lambda\sum_{k=1}^{N}u_k(n)\hat{\alpha}(k)\hat{f}(k)+\mu\sum_{k=1}^{N}u_k(n)\hat{\beta}(k)\hat{f}(k) \\
				&=\lambda(\alpha*f)(n)+\mu(\beta*f)(n)=(\lambda(\alpha*f)+\mu(\beta*f))(n).
			\end{align*}
			\item For all $k\in\lbrace1,\dots,N\rbrace$, by the orthonormality relations,
			\begin{align*}
				\widehat{\alpha*f}(k)&=\sum_{n=1}^{N}\overline{u_k(n)}(\alpha*f)(n)=\sum_{n=1}^{N}\overline{u_k(n)}\sum_{h=1}^{N}u_h(n)\hat{\alpha}(h)\hat{f}(h) \\
				&=\sum_{h=1}^{N}\sum_{n=1}^{N}u_h(n)\overline{u_k(n)}\hat{\alpha}(h)\hat{f}(h)=\sum_{h=1}^{N}\delta_{hk}\hat{\alpha}(h)\hat{f}(h)=\hat{\alpha}(k)\hat{f}(k).
			\end{align*}
			\item For all $n\in\lbrace1,\dots,N\rbrace$,
			\begin{equation*}
				(\alpha*\beta)(n)=\sum_{k=1}^{N}u_k(n)\hat{\alpha}(k)\hat{\beta}(k)=\sum_{k=1}^{N}u_k(n)\hat{\beta}(k)\hat{\alpha}(k)=(\beta*\alpha)(n).
			\end{equation*}
			\item By part (b), for all $n\in\lbrace1,\dots,N\rbrace$,
			\begin{align*}
				((\alpha*\beta)*f)(n)&=\sum_{k=1}^{N}u_k(n)\widehat{\alpha*\beta}(k)\hat{f}(k)=\sum_{k=1}^{N}u_k(n)(\hat{\alpha}(k)\hat{\beta}(k))\hat{f}(k) \\
				&=\sum_{k=1}^{N}u_k(n)\hat{\alpha}(k)(\hat{\beta}(k)\hat{f}(k))=\sum_{k=1}^{N}u_k(n)\hat{\alpha}(k)\widehat{\beta*f}(k) \\
                &=(\alpha*(\beta*f))(n).
			\end{align*}
			\item For all $k\in\lbrace1,\dots,N\rbrace$, by the orthonormality relations,
			\begin{equation*}
				\hat{\varepsilon}(k)=\sum_{n=1}^{N}\overline{u_k(n)}\varepsilon(n)=\sum_{n=1}^{N}\sum_{h=1}^{N}\overline{u_k(n)}u_h(n)=\sum_{h=1}^{N}\sum_{n=1}^{N}u_h(n)\overline{u_k(n)}=\sum_{h=1}^{N}\delta_{hk}=1,
			\end{equation*}
			and so, for all $n\in\lbrace1,\dots,N\rbrace$,
			\begin{align*}
				(\varepsilon*f)(n)&=\sum_{k=1}^{N}u_k(n)\hat{\varepsilon}(k)\hat{f}(k)=\sum_{k=1}^{N}u_k(n)\hat{f}(k)=f(n). \qedhere
			\end{align*}
        \end{enumerate}
	\end{proof}
    
	\textbf{Remark:} Part (e) of Theorem \ref{Properties of convolution} states that the scalar signal $\varepsilon:=u_1+\cdots+u_N$ is the "identity" of the convolution operator. In particular, note that $\hat{\varepsilon}$ is the constant function $1$, as $\hat{\varepsilon}(k)=1$ for all $k\in\lbrace1,\dots,N\rbrace$.
    
    \subsection{Young's inequality for graph signals}
    One of the core results involving the convolution operator in classical harmonic analysis is Young's inequality. It states that, for signals $f\in L^p(\mathbb{R}^d)$ and $g\in L^q(\mathbb{R}^d)$, the convolution $f*g$ belongs to $L^r(\mathbb{R}^d)$ and satisfies
    \begin{equation*}
        \|f*g\|_r\leq\|f\|_p\|g\|_q,
    \end{equation*}
    provided that $p$, $q$, $r$ satisfy the relation
    \begin{equation*}
        \frac{1}{p}+\frac{1}{q}=1+\frac{1}{r}.
    \end{equation*}
    The following theorem establishes a similar inequality in the graph setting.
	
	\begin{Theorem} \label{Young's inequality}
        Fix an orthonormal basis $\mathcal{B}=\lbrace u_1,\dots,u_N\rbrace$ of $\mathbb{K}^N$, a Banach space $X$, $\alpha\in\mathbb{K}^N$, and $f\in X^N$. Then, for $s,t,w\in[1,\infty]$ such that
        \begin{equation*}
            \frac{1}{s}+\frac{1}{t}+\frac{1}{w}=1,
        \end{equation*}
        and, for $p,q,r\in[1,\infty]$,
		\begin{equation*}
			\|\alpha*f\|_r\leq\|U^*\|_{s,r}\|U\|_{p',t}\|U\|_{q',w}\|\alpha\|_p\|f\|_q.
		\end{equation*}
		Consequently, $*:L^p(V,\mathbb{K})\times L^q(V,X)\to L^r(V,X)$ is a continuous bilinear operator, and
		\begin{equation*}
			\|*\|\leq\min\left\lbrace\|U^*\|_{s,r}\|U\|_{p',t}\|U\|_{q',w} \ \middle| \ \frac{1}{s}+\frac{1}{t}+\frac{1}{w}=1\right\rbrace.
		\end{equation*}
    \end{Theorem}
	\begin{proof}
		By Theorems \ref{Continuity of Fourier transform 3} and \ref{Properties of convolution}, and, by generalized Hölder's inequality,
        \begin{align*}
            \|\alpha*f\|_r&\leq\|U^*\|_{s,r}\|\widehat{\alpha*f}\|_{s'}=\|U^*\|_{s,r}\|\hat{\alpha}\hat{f}\|_{s'}\leq\|U^*\|_{s,r}\|\hat{\alpha}\|_t\|\hat{f}\|_w \\
            &\leq\|U^*\|_{s,r}\|U\|_{p',t}\|U\|_{q',w}\|\alpha\|_p\|f\|_q. \qedhere
        \end{align*}
	\end{proof}

    \section{Translation of vector-valued graph signals}
    In this section, we define the translation operator for graph signals, and investigate its analytical properties, such as continuity, and its algebraic properties, such as invertibility. We conclude by analyzing the action of the translation operator in the specific case where $X$ is a Hilbert space.
    
    The translation operator is one of the fundamental operators that can be defined on a vector space $X$, and its definition relies on the underlying algebraic structure. In classical time-frequency analysis, it is common to extend the translation to $L^p$ spaces and view its action as a linear operator from $L^p$ to itself, yielding the classical definition for $u\in\mathbb{R}^d$:
    \begin{equation*}
        T_u:L^p(\mathbb{R}^d)\to L^p(\mathbb{R}^d),f\mapsto T_uf, \qquad T_uf(x):=f(x-u).
    \end{equation*}
    In the distributional sense, the translation operator satisfies the following equality:
    \begin{equation*}
        T_uf(x)=(\delta_u*f)(x)=\int_{\mathbb{R}^d}e^{2\pi i\langle x,\xi\rangle}e^{-2\pi i\langle u,\xi\rangle}\hat{f}(\xi)d\xi,
    \end{equation*}
    where $\delta_u$ is the delta distribution centered at $u$.

    This spectral characterization is particularly useful in the graph setting, where the lack of a group structure prevents a direct definition via spatial shifts. Hence, we introduce the delta signal $\delta_m\in\mathbb{K}^N$ centered at vertex $m$:
    \begin{equation*}
		\delta_m(n):=
		\begin{cases}
			1 & \text{if } n=m,\\
			0 & \text{if } n\neq m.
		\end{cases}
	\end{equation*}
	Notice that, for all $k\in\lbrace1,\dots,N\rbrace$,
	\begin{equation*}
		\widehat{\delta_m}(k)=\sum_{n=1}^{N}\overline{u_k(n)}\delta_m(n)=\sum_{n=1}^{N}\overline{u_k(n)}\delta_{nm}=\overline{u_k(m)}.
	\end{equation*}

    \subsection{Translation of graph signals and main properties}
    Following \cite{Shuman_1}, the translation operator centered at vertex $m$ for graph signals is defined as the convolution between the delta signal $\delta_m$ and the signal $f$. We extend this to the setting of vector-valued graph signals as follows:
    \begin{equation*}
		T_m:X^N\to X^N,f\mapsto T_mf:=\delta_m*f, \qquad T_mf(n):=\sum_{k=1}^{N}u_k(n)\overline{u_k(m)}\hat{f}(k).
	\end{equation*}
    Note that $T_mf$ remains a vector-valued signal.
	
	\begin{Theorem} \label{Properties of translation}
        Fix an orthonormal basis $\mathcal{B}=\lbrace u_1,\dots,u_N\rbrace$ of $\mathbb{K}^N$, a Banach space $X$, $\alpha\in\mathbb{K}^N$, and $f\in X^N$.
        \begin{enumerate}[label=(\alph*)]
            \item $\widehat{T_mf}(k)=\overline{u_k(m)}\hat{f}(k)$ for all $k\in\lbrace1,\dots,N\rbrace$.
            \item $T_m(\alpha*f)=(T_m\alpha)*f=\alpha*(T_mf)$.
            \item $T_m\circ T_n=T_n\circ T_m$.
        \end{enumerate}
    \end{Theorem}
	\begin{proof}
		We prove the three properties separately.
		\begin{enumerate}[label=(\alph*), leftmargin=*, widest=a, font=\bfseries]
            \item By Theorem \ref{Properties of convolution},
            \begin{equation*}
                \widehat{T_mf}(k)=\widehat{\delta_m*f}(k)=\widehat{\delta_m}(k)\hat{f}(k)=\overline{u_k(m)}\hat{f}(k).
            \end{equation*}
			\item By the associativity and commutativity of convolution (Theorem \ref{Properties of convolution}), we have
			\begin{equation*}
				T_m(\alpha*f)=\delta_m*(\alpha*f)=(\delta_m*\alpha)*f=(T_m\alpha)*f,
			\end{equation*}
			and
			\begin{equation*}
                (T_m\alpha)*f=(\delta_m*\alpha)*f=(\alpha*\delta_m)*f=\alpha*(\delta_m*f)=\alpha*(T_mf).
			\end{equation*}
			\item By Theorem \ref{Properties of convolution}, for all $f\in X^N$,
			\begin{align*}
				(T_m\circ T_n)f&=\delta_m*(T_nf)=\delta_m*(\delta_n*f)=(\delta_m*\delta_n)*f=(\delta_n*\delta_m)*f \\
                &=\delta_n*(\delta_m*f)=\delta_n*(T_mf)=(T_n\circ T_m)f. \qedhere
			\end{align*}
		\end{enumerate}
	\end{proof}
    
	\begin{Theorem} \label{Convolution as translation}
       Fix an orthonormal basis $\mathcal{B}=\lbrace u_1,\dots,u_N\rbrace$ of $\mathbb{K}^N$, a Banach space $X$, $\alpha\in\mathbb{K}^N$, and $f\in X^N$. Then, for all $n\in\lbrace1,\dots,N\rbrace$,
		\begin{equation*}
			(\alpha*f)(n)=\sum_{m=1}^{N}\alpha(m)T_mf(n).
		\end{equation*}
    \end{Theorem}
	\begin{proof}
		By the definition of convolution and the graph Fourier transform, we have
        \begin{align*}
            (\alpha*f)(n)&=\sum_{k=1}^{N}u_k(n)\hat{\alpha}(k)\hat{f}(k)=\sum_{k=1}^{N}u_k(n)\sum_{m=1}^{N}\overline{u_k(m)}\alpha(m)\hat{f}(k) \\
            &=\sum_{m=1}^{N}\alpha(m)\sum_{k=1}^{N}u_k(n)\overline{u_k(m)}\hat{f}(k)=\sum_{m=1}^{N}\alpha(m)(\delta_m*f)(n) \\
            &=\sum_{m=1}^{N}\alpha(m)T_mf(n). \qedhere
        \end{align*}
	\end{proof}
	
	\textbf{Remark:} Via the integral formulation introduced, we may write
	\begin{equation*}
		(\alpha*f)(n)=\int_{V}\alpha(m)T_mf(n)dm.
	\end{equation*}
	This formulation is analogous to the classical definition of the convolution operator on $\mathbb{R}^d$ and, more generally, on locally compact abelian (LCA) groups.
	
	\begin{Theorem} \label{Continuity of translation}
        Fix an orthonormal basis $\mathcal{B}=\lbrace u_1,\dots,u_N\rbrace$ of $\mathbb{K}^N$, a Banach space $X$, and $f\in X^N$. Then, for $s,t,w\in[1,\infty]$ such that
        \begin{equation*}
            \frac{1}{s}+\frac{1}{t}+\frac{1}{w}=1,
        \end{equation*}
        and, for $p,q\in[1,\infty]$,
		\begin{equation*}
			\|T_mf\|_q\leq\|U^*\|_{s,q}\|U\|_{\infty,t}\|U\|_{p',w}\|f\|_p.
		\end{equation*}
		Consequently, $T_m:L^p(V,X)\to L^q(V,X)$ is a continuous linear operator, and
        \begin{equation*}
			\|T_m\|_{p\to q}\leq\min\left\lbrace\|U^*\|_{s,q}\|U\|_{\infty,t}\|U\|_{p',w} \ \middle| \ \frac{1}{s}+\frac{1}{t}+\frac{1}{w}=1\right\rbrace.
		\end{equation*}
    \end{Theorem}
	\begin{proof}
        Notice that $\|\delta_m\|_1=1$. Hence, by Theorem \ref{Young's inequality},
        \begin{equation*}
			\|T_mf\|_q=\|\delta_m*f\|_q\leq\|U^*\|_{s,q}\|U\|_{\infty,t}\|U\|_{p',w}\|f\|_p. \qedhere
		\end{equation*}
	\end{proof}

    \subsection{Kernel, range and invertibility of the translation}
    We now examine the algebraic structure of the translation operator. Specifically, we characterize its kernel and range, and establish the conditions under which the operator is invertible.
    
	\begin{Theorem} \label{Invertibility of translation}
        Fix an orthonormal basis $\mathcal{B}=\lbrace u_1,\dots,u_N\rbrace$ of $\mathbb{K}^N$, a Banach space $X$, and $m\in\lbrace1,\dots,N\rbrace$.
        \begin{enumerate}[label=(\alph*)]
            \item $\ker T_m=\lbrace f\in X^N \mid \hat{f}(k)=0 \text{ for all } k\in\lbrace1,\dots,N\rbrace \text{ such that } u_k(m)\neq0\rbrace$.
            \item $\im T_m=\lbrace g\in X^N \mid \hat{g}(k)=0 \text{ for all } k\in\lbrace1,\dots,N\rbrace \text{ such that } u_k(m)=0\rbrace$.
            \item $T_m$ is invertible $\Leftrightarrow u_k(m)\neq0$ for all $k\in\lbrace1,\dots,N\rbrace$ and, if $T_m$ is invertible, its inverse is the linear operator $T_m^{-1}$ defined by:
            \begin{equation*}
                T_m^{-1}:X^N\to X^N,g\mapsto T_m^{-1}g, \qquad T_m^{-1}g(n):=\sum_{k=1}^{N}\frac{u_k(n)}{\overline{u_k(m)}}\hat{g}(k).
            \end{equation*}
            \item $T_m$ is injective $\Leftrightarrow T_m$ is invertible $\Leftrightarrow T_m$ is surjective.
        \end{enumerate}
    \end{Theorem}
	\begin{proof}
        We prove the four properties separately.
		\begin{enumerate}[label=(\alph*), leftmargin=*, widest=a, font=\bfseries]
            \item Since the graph Fourier transform is invertible, by Theorem \ref{Properties of translation}, we have
            \begin{equation*}
                f\in\ker T_m \Leftrightarrow T_mf=0 \Leftrightarrow \widehat{T_mf}=0 \Leftrightarrow \overline{u_k(m)}\hat{f}(k)=0 \ \text{for all } k\in\lbrace1,\dots,N\rbrace.
            \end{equation*}
            Hence, if $f\in\ker T_m$, then $\hat{f}(k)=0$ for all $k\in\lbrace1,\dots,N\rbrace$ such that $u_k(m)\neq 0$. If $\hat{f}(k)=0$ for all $k\in\lbrace1,\dots,N\rbrace$ such that $u_k(m)\neq 0$, then $f\in\ker T_m$.
            \item We know that $g\in\im T_m$ if and only if there exists $f\in X^N$ such that $T_mf=g$. Since the graph Fourier transform is invertible, by Theorem \ref{Properties of translation}, we have
            \begin{equation*}
                T_mf=g \Leftrightarrow \widehat{T_mf}=\hat{g} \Leftrightarrow \overline{u_k(m)}\hat{f}(k)=\hat{g}(k) \ \text{for all } k\in\lbrace1,\dots,N\rbrace.
            \end{equation*}
            Hence, if $g\in\im T_m$, then $\hat{g}(k)=0$ for all $k\in\lbrace1,\dots,N\rbrace$ such that $u_k(m)=0$. If $\hat{g}(k)=0$ for all $k\in\lbrace1,\dots,N\rbrace$ such that $u_k(m)=0$, we can define the signal $f\in X^N$ such that, for all $k\in\lbrace1,\dots,N\rbrace$,
            \begin{equation*}
                \hat{f}(k):=
                \begin{cases}
                    \overline{u_k(m)^{-1}}\hat{g}(k) & \text{if } u_k(m)\neq0, \\
                    0 & \text{if } u_k(m)=0.
                \end{cases}
            \end{equation*}
            Hence, by Theorem \ref{Properties of translation},
            \begin{equation*}
                \widehat{T_mf}(k)=\overline{u_k(m)}\hat{f}(k)=\hat{g}(k),
            \end{equation*}
            for all $k\in\lbrace1,\dots,N\rbrace$. Thus, by the inversion formula, we obtain $T_mf=g$ and so $g\in\im T_m$.
            \item Suppose that $u_k(m)\neq0$ for all $k\in\lbrace1,\dots,N\rbrace$. Part (a) yields $\ker T_m=\lbrace0\rbrace$, proving the injectivity of $T_m$. Now we show its surjectivity. Fix $g\in X^N$, and define the signal $f\in X^N$ such that, for all $k\in\lbrace1,\dots,N\rbrace$,
            \begin{equation*}
                \hat{f}(k):=\frac{1}{\overline{u_k(m)}}\hat{g}(k).
            \end{equation*}
            Hence, by the inversion formula,
            \begin{equation*}
                f(n)=\sum_{k=1}^{N}u_k(n)\hat{f}(k)=\sum_{k=1}^{N}\frac{u_k(n)}{\overline{u_k(m)}}\hat{g}(k)=T_m^{-1}g(n),
            \end{equation*}
            for all $n\in\lbrace1,\dots,N\rbrace$, and, by Theorem \ref{Properties of translation},
            \begin{equation*}
                \widehat{T_mf}(k)=\overline{u_k(m)}\hat{f}(k)=\hat{g}(k),
            \end{equation*}
            for all $k\in\lbrace1,\dots,N\rbrace$. Thus, by the inversion formula, we obtain $T_mf=g$.

            Now we suppose that there exists $k^*\in\lbrace1,\dots,N\rbrace$ such that $u_{k^*}(m)=0$ and we show that $T_m$ is not invertible. Fix $x_0\in X-\lbrace0\rbrace$, and consider $f\in X^N-\lbrace0\rbrace$ such that $\hat{f}:=(0,\dots,x_0,\dots,0)$, having $x_0$ in the $k^*$-th position. Thus, by part (a), $f\in\ker T_m$ and so $T_m$ is not invertible.
            \item We are only left to show that, if $T_m$ is not invertible, then $T_m$ is neither injective nor surjective. Indeed, by part (c), suppose that there exists $k^*\in\lbrace1,\dots,N\rbrace$, such that $u_{k^*}(m)=0$. Fix $x_0\in X-\lbrace0\rbrace$, and consider the signal $f\in X^N-\lbrace0\rbrace$, such that $\hat{f}:=(0,\dots,x_0,\dots,0)$, having $x_0$ in the $k^*$-th position. Thus, by part (a) and part (b), $f\in\ker T_m$ and $f\notin\im T_m$. \qedhere
        \end{enumerate}
	\end{proof}
    
    \begin{Theorem} \label{Closedeness of kernel and range of the translation}
        Fix an orthonormal basis $\mathcal{B}=\lbrace u_1,\dots,u_N\rbrace$ of $\mathbb{K}^N$, a Banach space $X$, and $m\in\lbrace1,\dots,N\rbrace$.
        \begin{enumerate}[label=(\alph*)]
            \item $\ker T_m$ and $\im T_m$ are closed subspaces of $X^N$, and hence Banach spaces.
            \item If $K_0:=\lbrace k\in\lbrace1,\dots,N\rbrace \mid u_k(m)=0\rbrace$ and $d:=|K_0|$ is its cardinality, then $\ker T_m\cong X^d$ and $\im T_m\cong X^{N-d}$, where $\cong$ denotes topological isomorphism.
            \item $X^N=\ker T_m\oplus\im T_m$, as a topological direct sum.
            \item If $X$ is a Hilbert space, then $\ker T_m\cong X^d$ and $\im T_m\cong X^{N-d}$, where in this case $\cong$ denotes isometric isomorphism with respect to $L^2$ norms.
        \end{enumerate}
    \end{Theorem}
    \begin{proof}
        We prove the four properties separately.
		\begin{enumerate}[label=(\alph*), leftmargin=*, widest=a, font=\bfseries]
            \item By Theorem \ref{Continuity of translation}, $T_m$ is a continuous linear operator, hence $\ker T_m=T_m^{-1}(0)$ is a closed subspace of $X^N$. Let $(g_n)\subseteq\im T_m$ be a sequence such that $g_n\to g$ as $n\to\infty$, and fix an index $k\in\lbrace1,\dots,N\rbrace$ such that $u_k(m)=0$. So, by Theorem \ref{Invertibility of translation}, $\hat{g}_n(k)=0$ for all $n\in\mathbb{N}$, and, by continuity of the graph Fourier transform, $\hat{g}_n\to\hat{g}$ as $n\to\infty$. Thus, $\hat{g}(k)=0$, and hence $g\in\im T_m$, implying that $\im T_m$ is closed.
            \item Consider the following linear operator:
            \begin{equation*}
                \varphi:\ker T_m\to X^d,f\mapsto\varphi(f):=(\hat{f}(k))_{k\in K_0}.
            \end{equation*}
            Notice that, by Theorem \ref{Invertibility of translation}, $\varphi$ is an isomorphism, and its inverse is the linear operator $\varphi^{-1}$ defined by:
            \begin{equation*}
                \varphi^{-1}:X^d\to\ker T_m,g\mapsto\varphi^{-1}(g):=\sum_{k\in K_0}u_kg(k),
            \end{equation*}
            such that, for all $n\in\lbrace1,\dots,N\rbrace$,
            \begin{equation*}
                \varphi^{-1}(g)(n)=\sum_{k\in K_0}u_k(n)g(k)=\check{g}(n).
            \end{equation*}
            Since all $L^p$ norms are equivalent, we can prove the continuity of $\varphi$ and of $\varphi^{-1}$ with respect to any of them. By Theorem \ref{Continuity of Fourier transform 1}, for all $f\in\ker T_m$ and $g\in X^d$,
            \begin{equation*}
                \|\varphi(f)\|_\infty=\|\hat{f}\|_\infty\leq\kappa(U)\|f\|_1,
            \end{equation*}
            and
            \begin{equation*}
                \|\varphi^{-1}(g)\|_\infty=\|\check{g}\|_\infty\leq\kappa(U^*)\|g\|_1,
            \end{equation*}
            showing that $\varphi$ is a topological isomorphism. Analogously, consider the following linear operator:
            \begin{equation*}
                \psi:\im T_m\to X^{N-d},f\mapsto\psi(f):=(\hat{f}(k))_{k\in K_0^c}.
            \end{equation*}
            Notice that, by Theorem \ref{Invertibility of translation}, $\psi$ is an isomorphism, and its inverse is the linear operator $\psi^{-1}$ defined by:
            \begin{equation*}
                \psi^{-1}:X^{N-d}\to\im T_m,g\mapsto\psi^{-1}(g):=\sum_{k\in K_0^c}u_kg(k),
            \end{equation*}
            such that, for all $n\in\lbrace1,\dots,N\rbrace$,
            \begin{equation*}
                \psi^{-1}(g)(n)=\sum_{k\in K_0^c}u_k(n)g(k)=\check{g}(n).
            \end{equation*}
            Since all $L^p$ norms are equivalent, we can prove the continuity of $\psi$ and of $\psi^{-1}$ with respect to any of them. By Theorem \ref{Continuity of Fourier transform 1}, for all $f\in\im T_m$ and $g\in X^{N-d}$,
            \begin{equation*}
                \|\psi(f)\|_\infty=\|\hat{f}\|_\infty\leq\kappa(U)\|f\|_1,
            \end{equation*}
            and
            \begin{equation*}
                \|\psi^{-1}(g)\|_\infty=\|\check{g}\|_\infty\leq\kappa(U^*)\|g\|_1,
            \end{equation*}
            showing that $\psi$ is a topological isomorphism.
            \item Fix $f\in X^N$. By the inversion formula, for all $n\in\lbrace1,\dots,N\rbrace$, we have
            \begin{equation*}
                f(n)=\sum_{k=1}^{N}u_k(n)\hat{f}(k)=\sum_{k\in K_0}u_k(n)\hat{f}(k)+\sum_{k\in K_0^c}u_k(n)\hat{f}(k),
            \end{equation*}
            so that
            \begin{equation*}
                f=\sum_{k\in K_0}u_k\hat{f}(k)+\sum_{k\in K_0^c}u_k\hat{f}(k).
            \end{equation*}
            Hence, by Theorem \ref{Invertibility of translation},
            \begin{equation*}
                f_{\text{ker}}:=\sum_{k\in K_0}u_k\hat{f}(k)\in\ker T_m, \quad f_{\text{im}}:=\sum_{k\in K_0^c}u_k\hat{f}(k)\in\im T_m,
            \end{equation*}
            and $f=f_{\text{ker}}+f_{\text{im}}$. Note that, for $g\in\ker T_m\cap\im T_m$, by Theorem \ref{Invertibility of translation}, $\hat{g}(k)=0$ for all $k\in\lbrace1,\dots,N\rbrace$, and hence $g=0$. This demonstrates the algebraic direct sum $X^N=\ker T_m\oplus\im T_m$. Define the projection operators:
            \begin{equation*}
                P:X^N\to\ker T_m,f\mapsto Pf:=f_{\text{ker}}, \quad Q:X^N\to\im T_m,f\mapsto Qf:=f_{\text{im}}.
            \end{equation*}
            Note that $P$ and $Q$ are continuous linear operators. Indeed, by Theorem \ref{Continuity of Fourier transform 1},
            \begin{equation*}
                \|Pf\|_\infty=\|f_{\text{ker}}\|_\infty\leq\|\check{f}\|_\infty\leq\kappa(U^*)\|f\|_1,
            \end{equation*}
            and
            \begin{equation*}
                \|Qf\|_\infty=\|f_{\text{im}}\|_\infty\leq\|\check{f}\|_\infty\leq\kappa(U^*)\|f\|_1.
            \end{equation*}
            This demonstrates the topological direct sum $X^N=\ker T_m\oplus\im T_m$.
            \item Consider the isomorphisms defined before:
            \begin{equation*}
                \varphi:\ker T_m\to X^d,f\mapsto\varphi(f):=(\hat{f}(k))_{k\in K_0},
            \end{equation*}
            and
            \begin{equation*}
                \psi:\im T_m\to X^{N-d},f\mapsto\psi(f):=(\hat{f}(k))_{k\in K_0^c}.
            \end{equation*}
            Hence, by Plancherel's equality, for all $f\in\ker T_m$,
            \begin{equation*}
                \|\varphi(f)\|_2^2=\|(\hat{f}(k))_{k\in K_0}\|_2^2=\sum_{k\in K_0}\|\hat{f}(k)\|^2=\sum_{k=1}^{N}\|\hat{f}(k)\|^2=\|\hat{f}\|_2^2=\|f\|_2^2,
            \end{equation*}
            and, for all $f\in\im T_m$,
            \begin{equation*}
                \|\psi(f)\|_2^2=\|(\hat{f}(k))_{k\in K_0^c}\|_2^2=\sum_{k\in K_0^c}\|\hat{f}(k)\|^2=\sum_{k=1}^{N}\|\hat{f}(k)\|^2=\|\hat{f}\|_2^2=\|f\|_2^2.
            \end{equation*}
            This demonstrates that $\varphi$ and $\psi$ are isometric isomorphisms. \qedhere
        \end{enumerate}
	\end{proof}
    
    \begin{Corollary} \label{Induced translation operator}
        Fix an orthonormal basis $\mathcal{B}=\lbrace u_1,\dots,u_N\rbrace$ of $\mathbb{K}^N$, a Banach space $X$, and $m\in\lbrace1,\dots,N\rbrace$. Then, there exists a unique topological isomorphism
        \begin{equation*}
            \widetilde{T}_m:X^N/\ker T_m\to\im T_m,
        \end{equation*}
        such that $T_m=i\circ\widetilde{T}_m\circ\pi$, where $i:\im T_m\hookrightarrow X^N$ is the canonical injection, and
        \begin{equation*}
            \pi:X^N\to X^N/\ker T_m,f\mapsto\pi(f):=f+\ker T_m
        \end{equation*}
        is the canonical projection onto the quotient space, equipped with the norm
        \begin{equation*}
            \|f+\ker T_m\|:=\inf\lbrace\|f+z\| \mid z\in\ker T_m\rbrace.
        \end{equation*}
        In particular, $\|\widetilde{T}_m\|=\|T_m\|$.
    \end{Corollary}
    \begin{proof}
        By Theorem \ref{Closedeness of kernel and range of the translation}, $X^N=\ker T_m\oplus\im T_m$, as a topological direct sum, and thus $\widetilde{T}_m$ is well-defined by the first isomorphism theorem for Banach spaces. Moreover, for all $f\in X^N$ and $z\in\ker T_m$,
        \begin{equation*}
            \|\widetilde{T}_m(f+\ker T_m)\|=\|T_mf\|=\|T_m(f+z)\|\leq\|T_m\|\|f+z\|,
        \end{equation*}
        and so
        \begin{equation*}
            \|\widetilde{T}_m(f+\ker T_m)\|\leq\|T_m\|\|f+\ker T_m\|.
        \end{equation*}
        This shows that $\|\widetilde{T}_m\|\leq\|T_m\|$. Analogously, for all $f\in X^N$,
        \begin{equation*}
            \|T_mf\|=\|\widetilde{T}_m(f+\ker T_m)\|\leq\|\widetilde{T}_m\|\|f+\ker T_m\|\leq\|\widetilde{T}_m\|\|f\|,
        \end{equation*}
        demonstrating that $\|T_m\|\leq\|\widetilde{T}_m\|$. \qedhere
    \end{proof}
	
	\textbf{Remark:} The inverse induced translation operator $\widetilde{T}_m^{-1}$ acts as follows:
    \begin{equation*}
        \widetilde{T}_m^{-1}:\im T_m\to X^N/\ker T_m, \qquad \widetilde{T}_m^{-1}g=\sum_{k\in K_0^c}\frac{u_k}{\overline{u_k(m)}}\hat{g}(k)+\ker T_m,
    \end{equation*}
    where $K_0:=\lbrace k\in\lbrace1,\dots,N\rbrace \mid u_k(m)=0\rbrace$.

    \subsection{Translation operator of graph signals on Hilbert spaces}
    We now study further properties of the translation operator when $X$ is a Hilbert space.
    
	\begin{Theorem} \label{Translation on Hilbert spaces}
        Fix an orthonormal basis $\mathcal{B}=\lbrace u_1,\dots,u_N\rbrace$ of $\mathbb{K}^N$, a Hilbert space $X$, $m\in\lbrace1,\dots,N\rbrace$, and $f\in X^N$. Then,
        \begin{equation*}
            \|T_mf\|_2\leq\|u(m)\|_\infty\|f\|_2.
        \end{equation*}
        Consequently, $T_m:L^2(V,X)\to L^2(V,X)$ is a continuous linear operator, and
        \begin{equation*}
            \|\widetilde{T}_m\|_{2\to 2}=\max\lbrace|u_k(m)| \mid k\in K_0^c\rbrace, \quad \|\widetilde{T}_m^{-1}\|_{2\to 2}=\frac{1}{\min\lbrace|u_k(m)| \mid k\in K_0^c\rbrace},
		\end{equation*}
        where $K_0:=\lbrace k\in\lbrace1,\dots,N\rbrace \mid u_k(m)=0\rbrace$. Furthermore,
        \begin{equation*}
			\frac{\widetilde{T}_m}{\|u(m)\|_\infty} \text{ is an isometric isomorphism} \Leftrightarrow \|u(m)\|_\infty=(N-d)^{-1/2},
		\end{equation*}
        where $d:=|K_0|$.
    \end{Theorem}
	\begin{proof}
        By Plancherel's equality and by Theorem \ref{Properties of translation},
        \begin{align*}
			\|T_mf\|_2^2&=\|\widehat{T_mf}\|_2^2=\sum_{k=1}^{N}\|\widehat{T_mf}(k)\|^2=\sum_{k=1}^{N}\|\overline{u_k(m)}\hat{f}(k)\|^2 \\
            &=\sum_{k=1}^{N}|u_k(m)|^2\|\hat{f}(k)\|^2\leq\|u(m)\|_\infty^2\|\hat{f}\|_2^2=\|u(m)\|_\infty^2\|f\|_2^2.
		\end{align*}
        This estimate shows that $\|T_m\|_2\leq\|u(m)\|_\infty$. We need the reverse inequality. Fix $x_0\in X$, with $\|x_0\|=1$, and let $k^*\in\lbrace1,\dots,N\rbrace$ such that
        \begin{equation*}
            \|u(m)\|_\infty=|u_{k^*}(m)|.
        \end{equation*}
        Now consider the signal $f\in X^N$ such that $\hat{f}:=(0,\dots,x_0,\dots,0)$, having $x_0$ in the $k^*$-th position and $\|\hat{f}\|_2=\|x_0\|=1$. By Plancherel's equality and Theorem \ref{Properties of translation},
        \begin{align*}
            \frac{\|T_mf\|_2}{\|f\|_2}&=\frac{\|\widehat{T_mf}\|_2}{\|\hat{f}\|_2}=\left(\sum_{k=1}^{N}\|\widehat{T_mf}(k)\|^2\right)^{1/2}=\left(\sum_{k=1}^{N}|u_k(m)|^2\|\hat{f}(k)\|^2\right)^{1/2} \\
            &=|u_{k^*}(m)|=\|u(m)\|_\infty.
        \end{align*}
        Hence, we get the desired inequality. As $X$ is a Hilbert space $X^N/\ker T_m$ is isometrically isomorphic to $(\ker T_m)^\perp$, and, since $X^N=\ker T_m\oplus\im T_m$, as a topological direct sum, $X^N/\ker T_m$ is isometrically isomorphic to $\im T_m$. So, we can write the action of the induced translation operator $\widetilde{T}_m^{-1}$ as follows:
        \begin{equation*}
            \widetilde{T}_m^{-1}:\im T_m\to\im T_m, \qquad \widetilde{T}_m^{-1}g=\sum_{k\in K_0^c}\frac{u_k}{\overline{u_k(m)}}\hat{g}(k).
        \end{equation*}
        Notice that
        \begin{align*}
            \mathcal{F}(\widetilde{T}_m^{-1}g)(h)&=\sum_{n=1}^{N}\overline{u_h(n)}\widetilde{T}_m^{-1}g(n)=\sum_{n=1}^{N}\sum_{k\in K_0^c}\overline{u_h(n)}\frac{u_k(n)}{\overline{u_k(m)}}\hat{g}(k) \\
            &=\sum_{k\in K_0^c}\sum_{n=1}^{N}u_k(n)\overline{u_h(n)}\frac{1}{\overline{u_k(m)}}\hat{g}(k)=\sum_{k\in K_0^c}\delta_{hk}\frac{1}{\overline{u_k(m)}}\hat{g}(k) \\
            &=
            \begin{cases}
                \overline{u_h(m)^{-1}}\hat{g}(h) & \text{if } h\in K_0^c, \\
                0 & \text{if } h\in K_0,
            \end{cases}
        \end{align*}
        and so, by Plancherel's equality,
        \begin{align*}
            \|\widetilde{T}_m^{-1}g\|_2^2&=\|\mathcal{F}(\widetilde{T}_m^{-1}g)\|_2^2=\sum_{k=1}^{N}\|\mathcal{F}(\widetilde{T}_m^{-1}g)(k)\|^2=\sum_{k\in K_0^c}\frac{1}{|u_k(m)|^2}\|\hat{g}(k)\|^2 \\
            &\leq\frac{1}{\min\lbrace|u_k(m)| \mid k\in K_0^c\rbrace^2}\|\hat{g}\|_2^2=\frac{1}{\min\lbrace|u_k(m)| \mid k\in K_0^c\rbrace^2}\|g\|_2^2.
        \end{align*}
        This estimate shows that
        \begin{equation*}
            \|\widetilde{T}_m^{-1}\|_{2\to 2}\leq\frac{1}{\min\lbrace|u_k(m)| \mid k\in K_0^c\rbrace}.
		\end{equation*}
        We need the reverse inequality. Fix $x_0\in X$, with $\|x_0\|=1$, and $k^*\in K_0^c$ such that
        \begin{equation*}
            \min\lbrace|u_k(m)| \mid k\in K_0^c\rbrace=|u_{k^*}(m)|.
        \end{equation*}
        Now consider the signal $g\in\im T_m$, such that $\hat{g}:=(0,\dots,x_0,\dots,0)$, with $x_0$ in the $k^*$-th position and $\|\hat{g}\|_2=\|x_0\|=1$. By Plancherel's equality,
        \begin{align*}
            \frac{\|\widetilde{T}_m^{-1}g\|_2}{\|g\|_2}&=\frac{\|\mathcal{F}(\widetilde{T}_m^{-1}g)\|_2}{\|\hat{g}\|_2}=\left(\sum_{k=1}^{N}\|\mathcal{F}(\widetilde{T}_m^{-1}g)(k)\|^2\right)^{1/2} \\
            &=\left(\sum_{k\in K_0^c}\frac{1}{|u_k(m)|^2}\|\hat{g}(k)\|^2\right)^{1/2}=\frac{1}{|u_{k^*}(m)|}=\frac{1}{\min\lbrace|u_k(m)| \mid k\in K_0^c\rbrace}.
        \end{align*}
        Hence, we get the desired inequality.
        
        Now suppose $\|u(m)\|_\infty=(N-d)^{-1/2}$, then $|u_k(m)|=(N-d)^{-1/2}$ for all $k\in K_0^c$, and hence, by Plancherel's equality, for all $f\in\im T_m$,
        \begin{align*}
            \left\|\frac{\widetilde{T}_mf}{\|u(m)\|_\infty}\right\|_2^2&=\left\|\frac{\mathcal{F}(\widetilde{T}_mf)}{\|u(m)\|_\infty}\right\|_2^2=\sum_{k\in K_0^c}\frac{|u_k(m)|^2}{\|u(m)\|_\infty^2}\|\hat{f}(k)\|^2=\sum_{k\in K_0^c}\|\hat{f}(k)\|^2=\|f\|_2^2.
        \end{align*}
        Thus, $\widetilde{T}_m/\|u(m)\|_\infty$ is an isometric isomorphism. Conversely, suppose $\widetilde{T}_m/\|u(m)\|_\infty$ is an isometric isomorphism. Hence, its inverse
        \begin{equation*}
            \left(\frac{\widetilde{T}_m}{\|u(m)\|_\infty}\right)^{-1}=\|u(m)\|_\infty\widetilde{T}_m^{-1}
        \end{equation*}
        is also an isometric isomorphism, and thus
        \begin{equation*}
            1=\|\|u(m)\|_\infty\widetilde{T}_m^{-1}\|_{2\to 2}=\|u(m)\|_\infty\|\widetilde{T}_m^{-1}\|_{2\to 2}=\frac{\max\lbrace|u_k(m)| \mid k\in K_0^c\rbrace}{\min\lbrace|u_k(m)| \mid k\in K_0^c\rbrace}.
        \end{equation*}
        This demonstrates that $|u_k(m)|=\|u(m)\|_\infty$ for all $k\in K_0^c$, and thus, since $\mathcal{B}$ is an orthonormal basis, $\|u(m)\|_\infty=(N-d)^{-1/2}$. \qedhere
	\end{proof}
	
	\textbf{Remark:} In case where $\widetilde{T}_m/\|u(m)\|_\infty$ is an isometric isomorphism, Theorem \ref{Translation on Hilbert spaces} shows that $\|u(m)\|_\infty=(N-d)^{-1/2}$. Hence, the isometry is given by $(N-d)^{1/2}\widetilde{T}_m$. By setting $K_0^m:=\lbrace k\in\lbrace1,\dots,N\rbrace \mid u_k(m)=0\rbrace$ and $d_m:=|K_0^m|$, we have that
    \begin{equation*}
		\frac{\widetilde{T}_m}{\|u(m)\|_\infty} \text{ is an isometry for all } m \Leftrightarrow \|u(m)\|_\infty=(N-d_m)^{-1/2} \text{ for all } m.
	\end{equation*}
    In particular, if $T_m$ is invertible for all $m\in\lbrace1,\dots,N\rbrace$, by Theorem \ref{Invertibility of translation}, $u_k(m)\neq0$ for all $k,m\in\lbrace1,\dots,N\rbrace$, so $K_0^m=\emptyset$ for all $m\in\lbrace1,\dots,N\rbrace$, and thus
    \begin{equation*}
		\frac{T_m}{\|u(m)\|_\infty} \text{ is an isometric isomorphism for all } m \Leftrightarrow \kappa(U)=N^{-1/2}.
	\end{equation*}
    
	\begin{Theorem} \label{Adjoint operator of translation}
        Fix an orthonormal basis $\mathcal{B}=\lbrace u_1,\dots,u_N\rbrace$ of $\mathbb{K}^N$, a Hilbert space $X$, and $m\in\lbrace1,\dots,N\rbrace$. Then, the adjoint of $T_m$ is the following linear operator:
        \begin{equation*}
            T_m^*:X^N\to X^N,g\mapsto T_m^*g, \qquad T_m^*g(n):=\sum_{k=1}^{N}u_k(n)u_k(m)\hat{g}(k).
        \end{equation*}
    In particular, $T_m$ is self-adjoint if and only if $u_k(m)\in\mathbb{R}$ for all $k\in\lbrace1,\dots,N\rbrace$.
    \end{Theorem}
    \begin{proof}
        Notice that, for all $g\in X^N$ and $k\in\lbrace1,\dots,N\rbrace$,
        \begin{align*}
            \widehat{T_m^*g}(k)&=\sum_{n=1}^{N}\overline{u_k(n)}T_m^*g(n)=\sum_{n=1}^{N}\overline{u_k(n)}\sum_{h=1}^{N}u_h(n)u_h(m)\hat{g}(h) \\
            &=\sum_{h=1}^{N}\sum_{n=1}^{N}u_h(n)\overline{u_k(n)}u_h(m)\hat{g}(h)=\sum_{h=1}^{N}\delta_{hk}u_h(m)\hat{g}(h)=u_k(m)\hat{g}(k).
        \end{align*}
        Thus, by Parseval's identity and by Theorem \ref{Properties of translation}, for all $f,g\in X^N$,
        \begin{align*}
            \langle T_mf,g\rangle&=\langle\widehat{T_mf},\hat{g}\rangle=\sum_{k=1}^{N}\langle\widehat{T_mf}(k),\hat{g}(k)\rangle=\sum_{k=1}^{N}\langle\overline{u_k(m)}\hat{f}(k),\hat{g}(k)\rangle \\
            &=\sum_{k=1}^{N}\langle\hat{f}(k),u_k(m)\hat{g}(k)\rangle=\sum_{k=1}^{N}\langle\hat{f}(k),\widehat{T_m^*g}(k)\rangle=\langle\hat{f},\widehat{T_m^*g}\rangle=\langle f,T_m^*g\rangle.
        \end{align*}
        So, it is trivial to notice that, if $u_k(m)\in\mathbb{R}$ for all $k\in\lbrace1,\dots,N\rbrace$, then $T_m$ is self-adjoint. Therefore, we now suppose $T_m$ is self-adjoint and we show that $u_k(m)\in\mathbb{R}$, for all $k\in\lbrace1,\dots,N\rbrace$. In particular, we know that
        \begin{equation*}
            \sum_{k=1}^{N}u_k(n)\overline{u_k(m)}\hat{f}(k)=\sum_{k=1}^{N}u_k(n)u_k(m)\hat{f}(k),
        \end{equation*}
        for all $f\in X^N$ and $n\in\lbrace1,\dots,N\rbrace$. Fix an index $k\in\lbrace1,\dots,N\rbrace$, and $x_0\in X-\lbrace0\rbrace$. If $u_k(m)=0$, then $u_k(m)$ is real. If $u_k(m)\neq0$, then we consider the signal $f\in X^N$ such that $\hat{f}:=(0,\dots,x_0,\dots,0)$, having $x_0$ in the $k$-th position. Hence, for $n=m$, we have $|u_k(m)|^2x_0=u_k(m)^2x_0$, yielding $u_k(m)=\pm|u_k(m)|\in\mathbb{R}$. \qedhere
    \end{proof}
	
	\textbf{Remark:} As a consequence of Theorem \ref{Adjoint operator of translation}, we find that $T_m$ is self-adjoint for all $m\in\lbrace1,\dots,N\rbrace$ if and only if $\mathcal{B}$ consists of real-valued vectors, as generally happens when considering the orthonormal basis of eigenvectors of either the (normalized) Laplacian matrix or the adjacency matrix of an undirected graph.

    \section{Applications and numerical simulations}
    This section provides a series of numerical simulations demonstrating the application of the graph Fourier transform to vector-valued signals in the most common Banach spaces.
    
    \vspace{0.3cm}
    \textbf{Example:} Let $X=\mathbb{R}^3$ be the Banach space of $3$-dimensional vectors, equipped with the Euclidean norm, and consider the (undirected) circular graph with $N=10$ vertices. In Figure \ref{Signals with 3 components}, we represent the graph Fourier transform of the vector-valued signal $f:V\to\mathbb{R}^3$, whose spatial components $f_1$, $f_2$ and $f_3$ are defined as follows:
    \begin{align*}
        f_1&=(-1,-2,0,0,3,6,0,-3,2,1), \\
        f_2&=(-3,-1,-2,0,3,0,-4,0,1,0), \\
        f_3&=(0,1,3,4,-4,0,0,-1,2,-3).
    \end{align*}
    In the figure, the transformed signal's components are arranged from top to bottom and left to right.
    \begin{figure}[ht]
    	\centering
    	\includegraphics[width=0.3\linewidth]{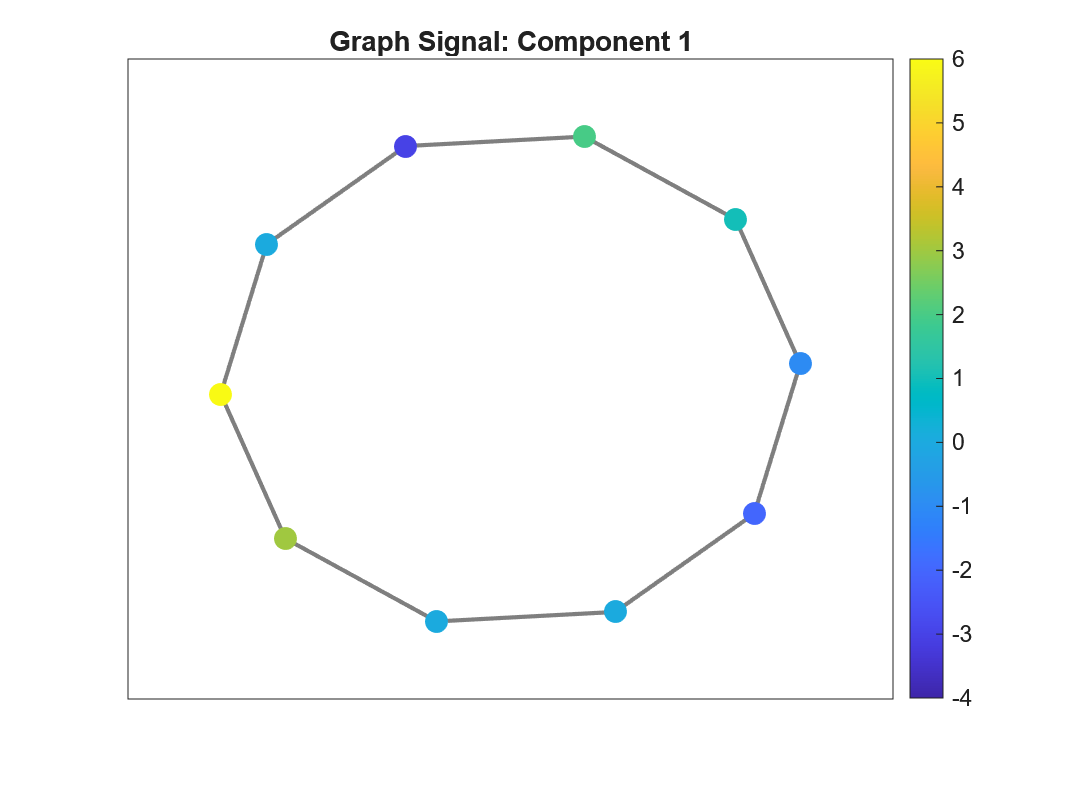}
    	\includegraphics[width=0.3\linewidth]{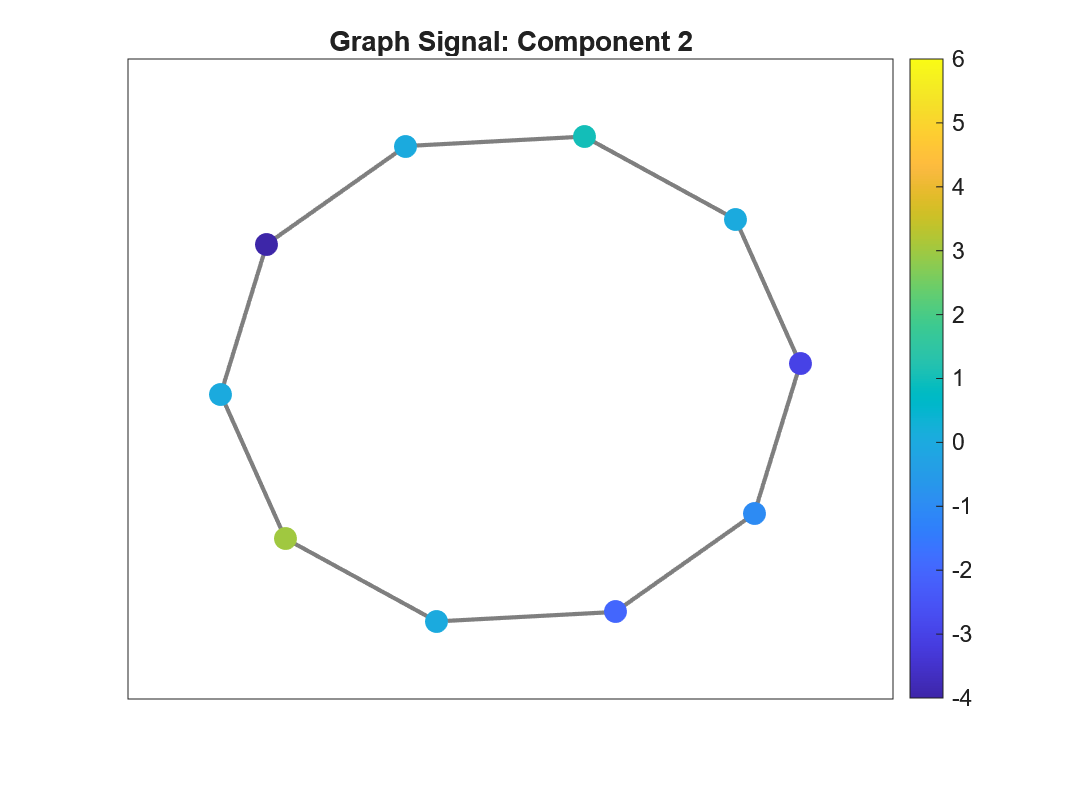}
    	\includegraphics[width=0.3\linewidth]{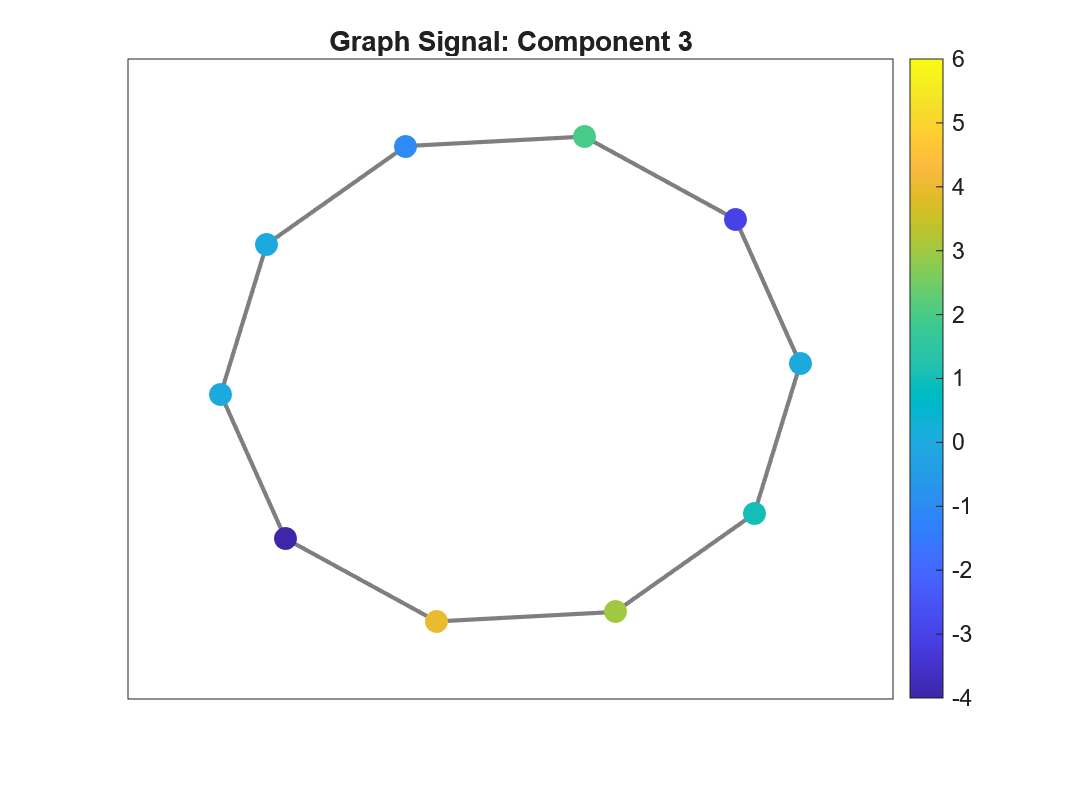} \\
    	\vspace{0.3cm}
    	\includegraphics[width=0.3\linewidth]{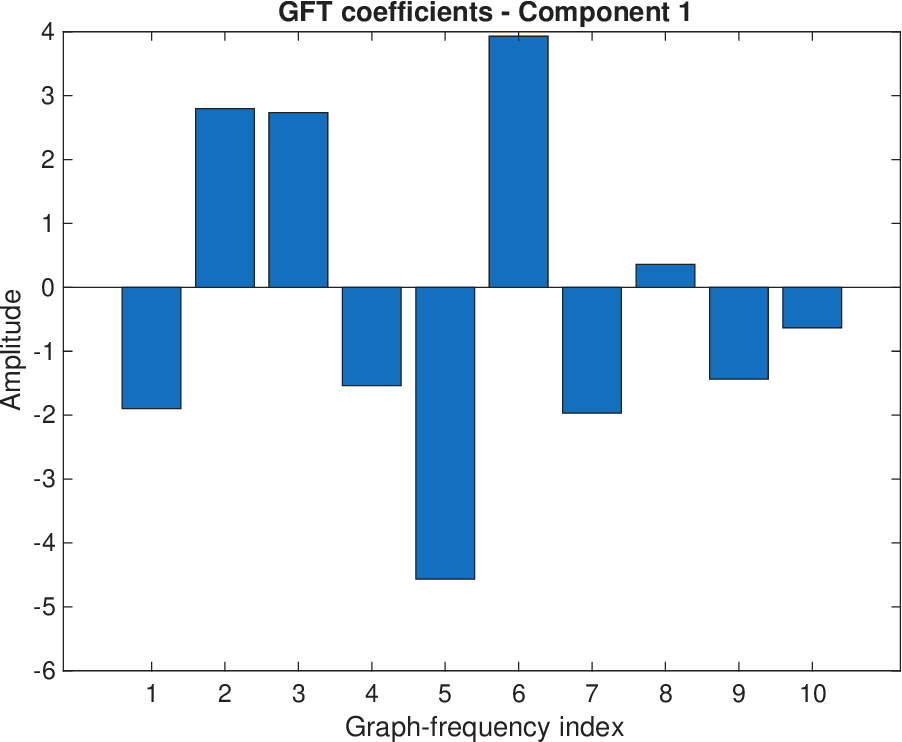}
    	\includegraphics[width=0.3\linewidth]{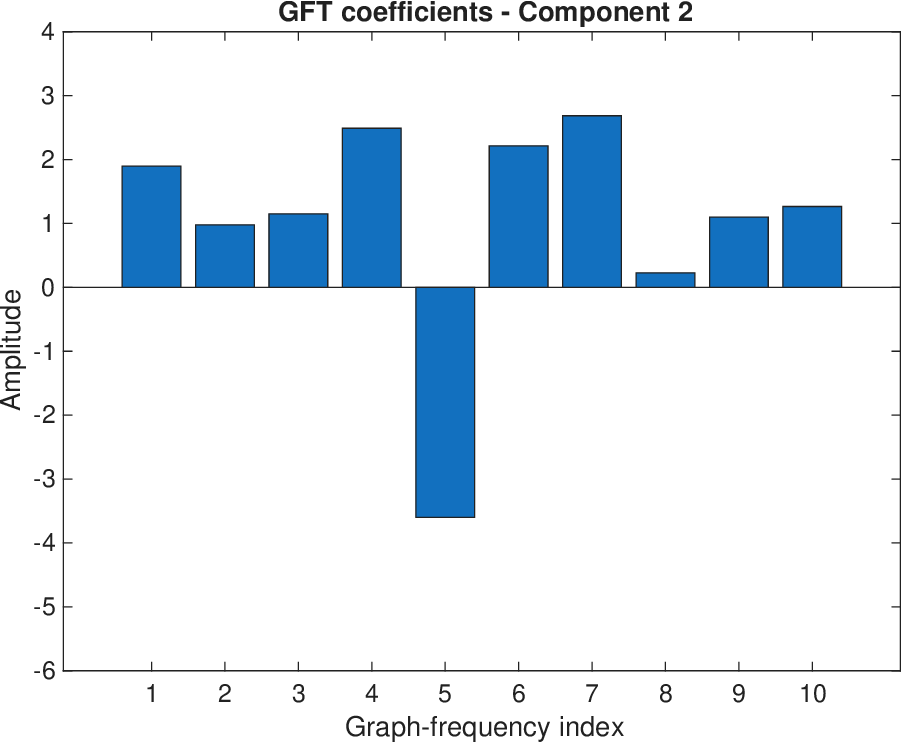}
    	\includegraphics[width=0.3\linewidth]{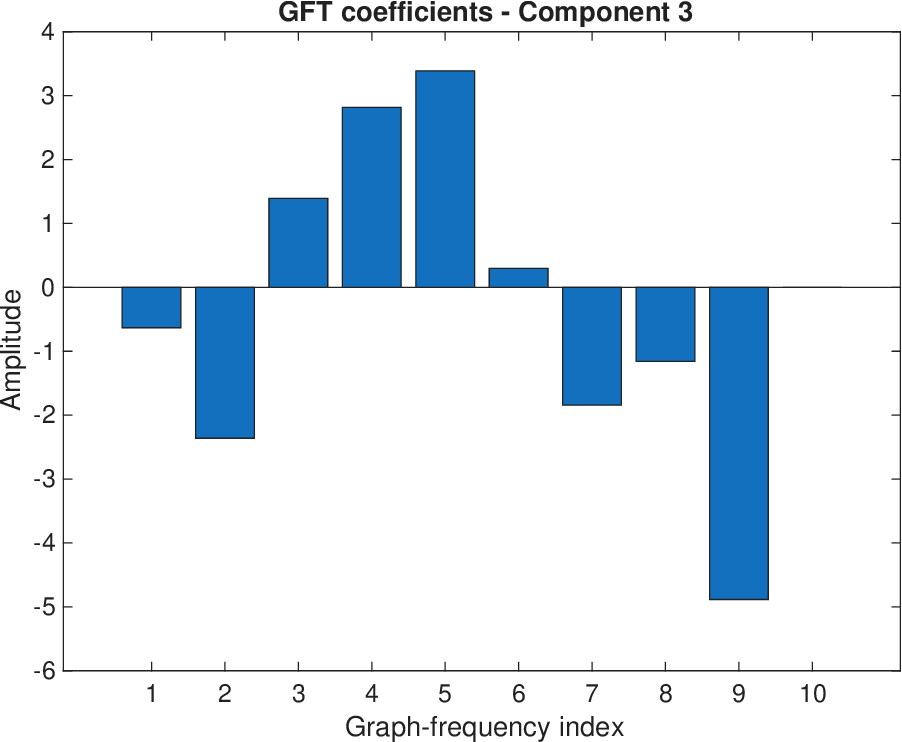} \\
    	\vspace{0.3cm}
    	\includegraphics[width=0.45\linewidth]{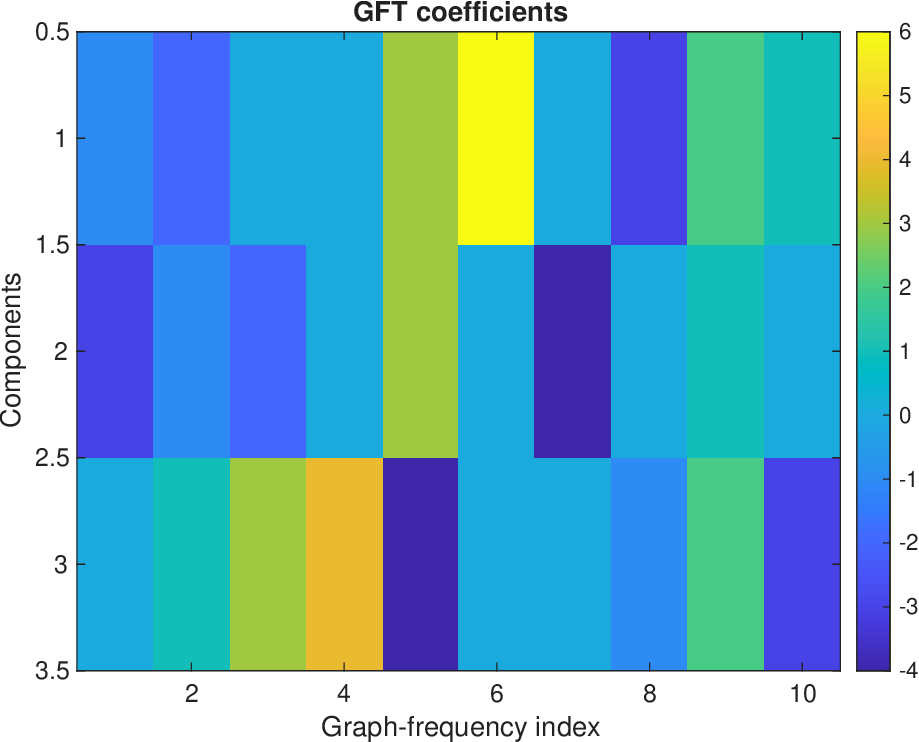}
    	\caption{Vector-valued signal on an undirected circular graph with $10$ vertices taking values into $\mathbb{R}^3$}
    	\label{Signals with 3 components}
    \end{figure}
    
    \textbf{Example:} Let $X=\mathcal{C}([0,100])\times\mathcal{C}([0,100])$ be the Banach space of continuous functions defined on the interval $[0,100]$, equipped with the $\infty$-norm, and consider the (undirected) circular graph with $10$ vertices. Figure \ref{Signals with functional components} displays the graph Fourier transform of the signal $f:V\to X$, whose components at each vertex $n$ are defined as:
    \begin{equation*}
        f_1(t,n)=\sin(t+0.3n), \quad f_2(t,n)=\cos(t+0.2n),
    \end{equation*}
    for $t\in[0,100]$. The functions are sampled at consecutive time instants with a step of $4\pi$ for visualization.
    \begin{figure}[ht]
    	\centering
    	\includegraphics[width=0.45\linewidth]{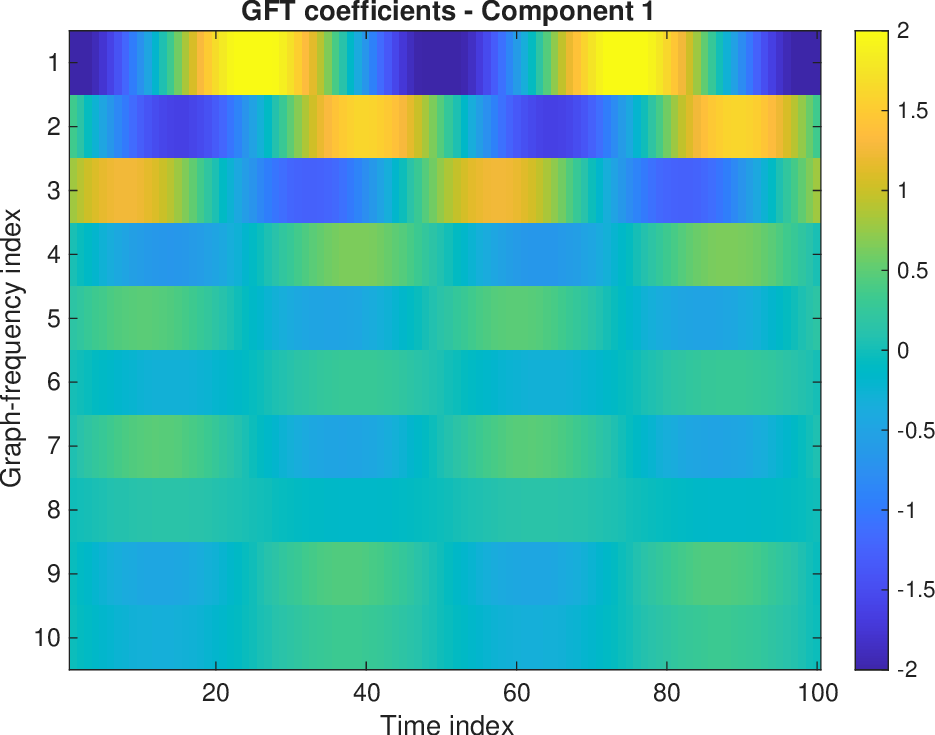}
    	\includegraphics[width=0.45\linewidth]{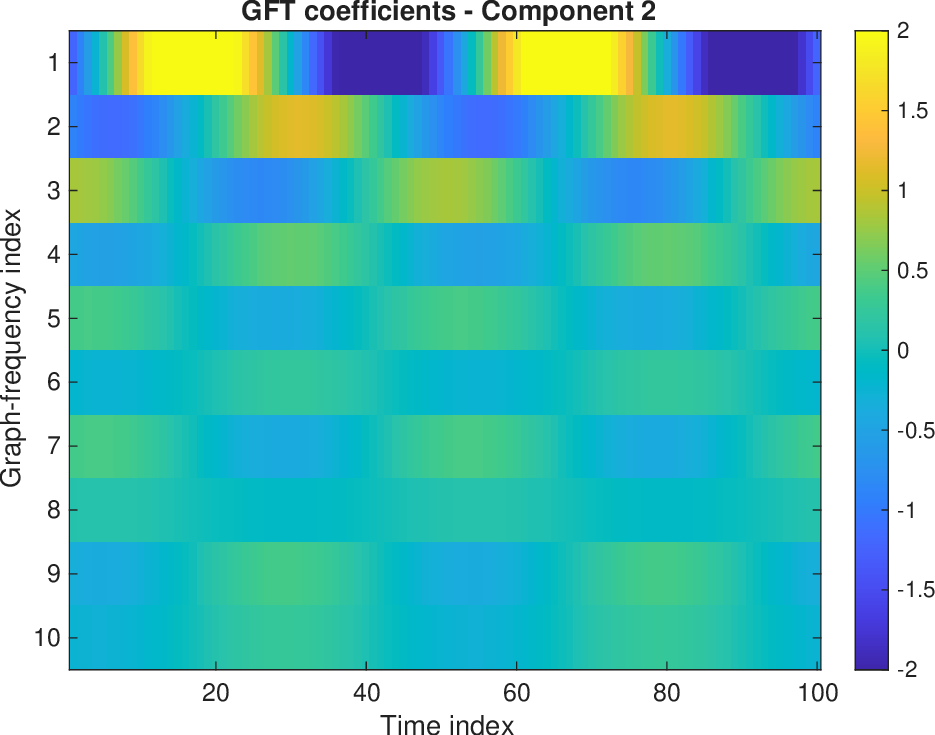}
    	\caption{Vector-valued signal on an undirected circular graph with $10$ vertices taking values into $\mathcal{C}({[0,100]})\times\mathcal{C}({[0,100]})$}
    	\label{Signals with functional components}
    \end{figure}
    
    \textbf{Example:} We aim to examine how the operator norms presented in Table \ref{Fourier_norms} vary depending on the choice of the orthonormal basis $\mathcal{B}$. To illustrate this, consider the (undirected) path graph with $N=4$ vertices, displayed in Figure \ref{Examples of graphs}. Its adjacency matrix, degree matrix, Laplacian matrix and normalized Laplacian matrix are given by:
    \begin{alignat*}{4}
        A&=
        \begin{pmatrix}
			0 & 1 & 0 & 0 \\
			1 & 0 & 1 & 0 \\
            0 & 1 & 0 & 1 \\
            0 & 0 & 1 & 0 \\
		\end{pmatrix}, \qquad
        &&D=
        \begin{pmatrix}
			1 & 0 & 0 & 0 \\
			0 & 2 & 0 & 0 \\
            0 & 0 & 2 & 0 \\
            0 & 0 & 0 & 1 \\
		\end{pmatrix}, \\
        L&=
        \begin{pmatrix*}[r]
			1 & -1 & 0 & 0 \\
			-1 & 2 & -1 & 0 \\
            0 & -1 & 2 & -1 \\
            0 & 0 & -1 & 1 \\
		\end{pmatrix*}, \qquad
        &&\mathcal{L}=
        \begin{pmatrix*}[r]
			1 & -\frac{1}{\sqrt{2}} & 0 & 0 \\
			-\frac{1}{\sqrt{2}} & 1 & -\frac{1}{2} & 0 \\
            0 & -\frac{1}{2} & 1 & -\frac{1}{\sqrt{2}} \\
            0 & 0 & -\frac{1}{\sqrt{2}} & 1 \\
		\end{pmatrix*}.
    \end{alignat*}
    Calculating the spectra of $A$, $L$ and $\mathcal{L}$, we obtain the following values (rounded to four decimal places):
    \begin{align*}
        \sigma(A)&=\lbrace-1.6180,-0.6180,0.6180,1.6180\rbrace, \\
        \sigma(L)&=\lbrace0,0.5858,2,3.4142\rbrace, \\
        \sigma(\mathcal{L})&=\lbrace0,0.5,1.5,2\rbrace.
    \end{align*}
    The associated unitary matrices, whose entries are rounded to four decimal places and whose columns are the normalized eigenvectors ordered according to the spectra above, are given by:
    \begin{align*}
        U_A&=
        \begin{pmatrix*}[r]
			0.3717 & -0.6015 & -0.6015 & 0.3717 \\
			-0.6015 & 0.3717 & -0.3717 & 0.6015 \\
            0.6015 & 0.3717 & 0.3717 & 0.6015 \\
            -0.3717 & -0.6015 & 0.6015 & 0.3717 \\
		\end{pmatrix*}, \\[2ex]
        U_L&=
        \begin{pmatrix*}[r]
			-0.5000 & 0.6533 & 0.5000 & -0.2706 \\
			-0.5000 & 0.2706 & -0.5000 & 0.6533 \\
            -0.5000 & -0.2706 & -0.5000 & -0.6533 \\
            -0.5000 & -0.6533 & 0.5000 & 0.2706 \\
		\end{pmatrix*}, \\[2ex]
        U_\mathcal{L}&=
        \begin{pmatrix*}[r]
			0.4082 & -0.5774 & 0.5774 & 0.4082 \\
			0.5774 & -0.4082 & -0.4082 & -0.5774 \\
            0.5774 & 0.4082 & -0.4082 & 0.5774 \\
            0.4082 & 0.5774 & 0.5774 & -0.4082 \\
		\end{pmatrix*}.
    \end{align*}
    For the sake of comparison, we also consider the unitary matrix given by normalized discrete Fourier transform:
    \begin{equation*}
        U_F=
        \begin{pmatrix*}[r]
			0.5000 & 0.5000 & 0.5000 & 0.5000 \\
			0.5000 & -0.5000i & -0.5000 & 0.5000i \\
            0.5000 & -0.5000 & 0.5000 & -0.5000 \\
            0.5000 & 0.5000i & -0.5000 & -0.5000i \\
		\end{pmatrix*}.
    \end{equation*}
    A comparison of the norms for the four different orthonormal bases is summarized in Table \ref{Fourier_norms_comparison}.
    \begin{table}[ht]
        \centering
        \caption{Operator norms $\kappa_p(U)$ for different values of $p$}
        \label{Fourier_norms_comparison}
        \renewcommand{\arraystretch}{1.5} 
        \begin{tabular}{cccccccc} 
            \toprule
            Operator & $\kappa_1(U)$  & $\kappa_{1.5}(U)$ & $\kappa_2(U)$ & $\kappa_3(U)$ & $\kappa_4(U)$ & $\kappa_{20}(U)$ & $\kappa_\infty(U)$ \\ \midrule
            $U_A$           & $1.9464$ & $1.3854$ & $1.0000$ & $0.8133$ & $0.7401$ & $0.6441$ & $0.6015$ \\
            $U_L$           & $2.0000$ & $1.5874$ & $1.0000$ & $0.8422$ & $0.7826$ & $0.6946$ & $0.6533$ \\
            $U_\mathcal{L}$ & $1.9712$ & $1.4081$ & $1.0000$ & $0.8047$ & $0.7260$ & $0.6139$ & $0.5774$ \\
            $U_F$           & $2.0000$ & $1.5874$ & $1.0000$ & $0.7937$ & $0.7071$ & $0.5359$ & $0.5000$ \\
            \bottomrule
            \end{tabular}
    \end{table}
    
    The results in Table \ref{Fourier_norms_comparison} illustrate that, while the operator norm is unitary for $p=2$, the choice of the orthonormal basis significantly influences the transform's behavior in other Banach spaces. Notably, the normalized Laplacian basis exhibits a closer alignment with the classical Fourier basis, in terms of its asymptotic decay towards $\kappa(U)$, whereas the adjacency-based transform maintains higher values across the mid-range $p$-norms.

    Furthermore, from a practical standpoint, lower values of $\kappa_p(U)$ are preferable for the operator norm estimates in Table \ref{Fourier_norms}, since they imply a more stable and bounded transformation. Conversely, higher values of $\kappa_p(U)$ enhance the lower bounds in the uncertainty principle (Theorems \ref{Uncertainty principle 1} and \ref{Uncertainty principle 2}), suggesting that graphs with higher spectral coherence offer stronger guarantees against simultaneous localization in both the vertex and spectral domains.

    This suggests that the selection of the operator must be carefully tuned to the specific geometric constraints of the Banach space in which the signal takes values and the desired balance between transform stability and uncertainty bounds.

    \section{Conclusion}
    This work established a generalized framework for signals defined on finite graphs and for standard graph-based data by introducing vector-valued graph signals. By defining the corresponding graph Fourier transform through an arbitrary orthonormal basis, we demonstrated a versatile approach to vertex-frequency analysis. Our derivation of operator norm estimates, and the fundamental properties of modulation and translation provide a rigorous mathematical foundation for this extension. Numerical simulations validated the theoretical results, particularly highlighting the impact of the $p$-coherence of the basis on signal sparsity. This general framework provides a flexible and mathematically grounded environment for graph signal processing, with significant potential for applications in time-series analysis and the study of time-varying signals on both networks and irregular domains.

    \section*{Acknowledgments}
    The author has been supported by the Gruppo Nazionale per l’Analisi Matematica, la Probabilità e le loro Applicazioni (GNAMPA) of the Istituto Nazionale di Alta Matematica (INdAM). He also acknowledges the support of Professor Iulia Martina Bulai for supervising his work and for her essential guidance.
    \printbibliography
	
\end{document}